\documentclass[12pt]{article}
\usepackage{geometry} 
\usepackage[utf8]{inputenc}
\usepackage[T1]{fontenc}
\usepackage[catalan,german,spanish,dutch,british]{babel}
\usepackage{lmodern}
\usepackage{graphicx}
\usepackage{amsmath}   %  Paquets per
\usepackage{amssymb}   %  matemàtics
\usepackage{amsfonts}
\usepackage{amsthm}
\usepackage{epstopdf}
\usepackage{mathtools}
\usepackage[pdftex]{color}
\usepackage{xcolor}
\usepackage{bm}
\usepackage{tikz}
\usepackage{url}
\usetikzlibrary{matrix}
\usepackage{multicol}
\usepackage{array}% http://ctan.org/pkg/array
\usepackage[parfill]{parskip}
\usepackage{fancyhdr}
\usepackage{lipsum}
\usepackage{longtable}
\usepackage{enumerate}
\usepackage{subcaption}
\usepackage{pinlabel}
\allowdisplaybreaks

\usepackage[pdftex]{color}
\usepackage{xcolor}

\usepackage[parfill]{parskip}
\makeatletter
\def\thm@space@setup{%
  \thm@preskip=\parskip \thm@postskip=0pt
}
\makeatother

\theoremstyle{definition}
\newtheorem{de}{Definition}

\theoremstyle{plain}
\newtheorem{theo}[de]{Theorem}
\newtheorem{pro}[de]{Proposition}

\newtheorem{co}[de]{Corollary}

\newtheorem{re}[de]{Remark}

\begingroup
    \makeatletter
    \@for\theoremstyle:=definition,remark,plain\do{%
        \expandafter\g@addto@macro\csname th@\theoremstyle\endcsname{%
            \addtolength\thm@preskip\parskip
            }%
        }
\endgroup

\newcommand{\vungoc}{V\~u Ng\d{o}c }
\newcommand{\nff}{{n_{\text{FF}}}}

\newcommand{\dee}{\mathrm{d}}

\newcommand{\R}{\mathbb{R}}
\newcommand{\mbS}{\mathbb{S}}
\newcommand{\T}{\mathbb{T}} % torus
\newcommand{\Z}{\mathbb{Z}}

\newcommand{\al}{\alpha}
\newcommand{\be}{\beta}
\newcommand{\symbga}{\gamma}

\newcommand{\ze}{\zeta}

\newcommand{\lam}{\lambda}

\newcommand{\om}{\omega}

\newcommand{\De}{\Delta}

\newcommand{\Lam}{\Lambda}

\newcommand{\Om}{\Omega}

\newcommand{\pti}{{\tilde{p}}}
\newcommand{\qti}{{\widetilde{q}}}

\newcommand{\mcA}{\mathcal A}
\newcommand{\mcB}{\mathcal B}

\newcommand{\mcF}{\mathcal F}

\newcommand{\mcH}{\mathcal H}

\newcommand{\mcL}{\mathcal L}

\newcommand{\mcN}{\mathcal N}

\newcommand{\mcP}{\mathcal P}

\newcommand{\mcX}{\mathcal X}

\title{\textbf{The height invariant of a four-parameter semitoric system with two focus-focus singularities}}
\author{Jaume Alonso \and Sonja Hohloch}
\date{\small 27th June 2020}

\begin{document}

\maketitle

\begin{abstract}
Semitoric systems are a special class of completely integrable systems with two degrees of freedom that have been symplectically classified by Pelayo and \vungoc about a decade ago in terms of five symplectic invariants. If a semitoric system has several focus-focus singularities, then some of these invariants have multiple components, one for each focus-focus singularity. Their computation is not at all evident, especially in multi-parameter families. In this paper, we consider a {\em four-parameter} family of semitoric systems with {\em two} focus-focus singularities. In particular, apart from the polygon invariant, we compute the so-called {\em height invariant}. Moreover, we show that the two components of this invariant encode the symmetries of the system in an intricate way.
\end{abstract}

%%%%%%%%%%%%%%%%%%%%%%%%%%%%%%%%%%%%%%%%%%%%%%%%%%%%%%%%%%%%%%%%%%%%%%%%%
%%%%%%%%%%%%%%  new section  %%%%%%%%%%%%%%%%%%%%%%%%%%%%%%%%%%%%%%%%%%%%
%%%%%%%%%%%%%%%%%%%%%%%%%%%%%%%%%%%%%%%%%%%%%%%%%%%%%%%%%%%%%%%%%%%%%%%%%%%

\section{Introduction}

In the last decades, various efforts have been made towards the construction of classifications within the theory of completely integrable dynamical systems. These classifications are based on invariants that capture various aspects of a system with respect to different notions of equivalence. They are useful for two main reasons: they give an overview of all possible systems within a certain class and allow us to distinguish between non-equivalent systems. If we restrict ourselves to classifications of {\em symplectic} type, important accomplishments are the classification of toric systems, due to Delzant \cite{Del}, Atiyah \cite{At} and Guillemin $\&$ Sternberg \cite{GS} and the classification of semitoric systems, due to Pelayo $\&$ \vungoc \cite{PV1, PV4} and recently extended by Palmer \emph{et al.}\ \cite{PPT}. Another significant result in this line is the symplectic classification of completely integrable systems using characteristic classes, introduced by Zung \cite{Z2}.

Semitoric systems are a class of dynamical systems defined on connected four-dimensional symplectic manifolds, introduced by \vungoc \cite{Vu2}. They are integrable systems, so they have two conserved quantities, one of which is a proper map that induces an effective circle action. Moreover, all singularities are required to be non-degenerate and must not have hyperbolic components. From a {\em topological} point of view, these systems can be described using the theory of singular Lagrangian fibrations, cf.\ Bolsinov $\&$ Fomenko \cite{BF}.

From the {\em symplectic} point of view, one of the motivations to study semitoric systems comes from the analysis of systems with monodromy in the quantum physics and chemistry literature, see for example Child \emph{et.\ al.\ }\cite{CWT}, Sadovksii $\&$ Zhilinskii \cite{SZ} for a theoretical approach and Ass\'emat \emph{et.\ al.\ }\cite{AEJ}, Fitch \emph{et.\ al.\ }\cite{FWP}, Winnewisser \emph{et.\ al.\ }\cite{WWM} for experimental studies.

In this setting, one has the joint spectrum of a set of unknown quantum operators and wants to recover information about the system. An overview of the possible candidate systems can be obtained by means of a classification. Since classical systems are generally easier to understand, one can make use of Bohr's correspondence principle or \emph{Zauberstab} and focus on constructing a classification for classical systems. However, in order for the results to be valid after quantisation, it is important that this classsification preserves the symplectic structure, cf.\ Pelayo \cite{Pe} for more details on this approach.

Two foundational examples of the semitoric systems theory are the coupled spin-oscillator and the coupled angular momenta. The first one is a particular case of the Jaynes-Cummings \cite{JC} model  from quantum optics and it consists of the coupling of a classical spin on the two-sphere $\mbS^2$ with a harmonic oscillator on the plane $\R^2$, see e.g.\ Pelayo $\&$ \vungoc \cite{PV3}. The second one is the classical version of the addition of two quantum angular momenta, defined on the product of two copies of $\mbS^2$. It models, for example, the reduced Hamiltonian of a hydrogen-like atom in the presence of parallel electric and magnetic fields, cf.\ Sadovskii \emph{et al.}\ \cite{SZM}. In the last years, several other examples of semitoric systems have been discovered: Hohloch $\&$ Palmer \cite{HP} introduced a family with two focus-focus points, Le Floch $\&$ Palmer \cite{LFPa} proved the existence of examples in all Hirzebruch surfaces and De Meulenaere $\&$ Hohloch \cite{DMH2} proposed a system with four focus-focus points that has double pinched focus-focus fibres for a certain value of the parameter.  

The classification of semitoric systems is based on five symplectic invariants: the number of focus-focus points, the polygon invariant, the height invariant, the Taylor series invariant and the twisting index invariant.

The survey article by Alonso $\&$ Hohloch \cite{AH} gives an overview of the state of the art concerning examples and computations of invariants reached in 2019. Note that the computation of these invariants is far from trivial, especially if the aim is to make a general calculation of the invariants for a {\em whole family} of systems {\em depending on several parameters}, instead of for only one explicitly given, concrete system.

So far, the {\em full list} of invariants has only been computed for the {\em two foundational examples}. The computation of the invariants in these two cases is based on the use of the properties of elliptic integrals. In the case of the coupled spin-oscillator, it was initiated by Pelayo $\&$ \vungoc \cite{PV3} and completed by Alonso $\&$ Dullin $\&$ Hohloch \cite{ADH}. In this case, two parameters are taken into account, but the dependence is quite simple. For the coupled angular momenta, it was initiated by Le Floch $\&$ Pelayo \cite{LFPe} and completed by Alonso $\&$ Dullin $\&$ Hohloch \cite{ADH2}. In this case, the dependence is of three parameters and significantly more involved.

Expressing the invariants as a function of the parameters of the system is important because, besides the quantitative results, it also allows for qualitative considerations. For instance, one can compare the roles played by \emph{geometric parameters}, i.e.\ those related to the symplectic manifold, and by \emph{coupling parameters}, i.e.\ those only appearing in the momentum map. In case some parameters also affect the type of singularities, for example making focus-focus singularities appear and disappear, one can also see what happens to the invariants as the critical values of the parameters are approached.

In both foundational examples, the invariants display the symmetries of the systems. Moreover, for the coupled angular momenta, the terms of the Taylor series invariant go to infinity as the coupling parameter approaches the critical values. However, a limitation of these examples is that the number of focus-focus points is at most one. Semitoric systems with more than one focus-focus point are interesting because, in this case, the symplectic invariants have multiple components, one for each focus-focus point. So the different components can (and should) be compared with each other. In particular, it is interesting to see how the different components depend on the parameters of the system and how they reflect the possible symmetries of the system. 

Note that the presence of multiple focus-focus points increases the complexity of the computations significantly. So far, the only results in this direction are the computation of the polygon invariant and the height invariant of two families of systems with a relatively simple dependence on two parameters, cf.\ Le Floch $\&$ Palmer \cite{LFPa}. There is, in general, a certain trade-off between, on the one hand, the qualitative richness of having the invariants expressed as functions of several parameters and, on the other hand, the feasability of their computations.

In the present paper we choose the former option, i.e., focusing on dependence on multiple parameters. We managed to compute the number of focus-focus points invariant, the polygon invariant, and the height invariant, but, given the number of parameters, the computational complexity of the Taylor series invariant and the twisting index invariant is beyond current computational methods and resources.

Let $(M,\om)$ be the symplectic manifold $M=\mbS^2 \times \mbS^2$ with symplectic form $\omega = - (R_1\, \omega_{\mbS^2}  \oplus  R_2\, \omega_{\mbS^2})$, where $\omega_{\mbS^2}$ is the standard symplectic form of the unit sphere and $0<R_1 <R_2$ or $0< R_2< R_1$. 
Given $(x_1,y_1,z_1, x_2,y_2,z_2)$ Cartesian coordinates in $M$ and $s_1, s_2 \in [0,1]$, we consider the integrable system $(M,\om,F)$, where $F:=(L,H)$ is defined by
\begin{equation}
\begin{cases}
L(x_1,y_1,z_1,x_2,y_2,z_2)\;:=R_1 z_1  + R_2 z_2, \\
H(x_1,y_1,z_1,x_2,y_2,z_2):=(1-2s_1)(1-s_2) z_1 + (1-2s_1)s_2 z_2 
\\ \hspace{4.4cm}+ 2(s_1+s_2-{s_1}^2 - {s_2}^2) (x_1x_2 + y_1y_2).\\
\end{cases}
\label{sysdefintro}
\end{equation}

It is a family of semitoric systems that can have up to two focus-focus singularities and depends on four parameters in total, two geometric parameters $R_1,R_2>0,$ $R_1 \neq R_2$ and two coupling parameters $s_1,s_2 \in [0,1]$. 

Our first result is the computation of the number of focus-focus singularities:

\begin{theo} The number of focus-focus points invariant of system \eqref{sysdefintro} is $\nff =0$ if $E>0$ and $\nff=2$ if $E<0$, where \begin{equation*}
\begin{aligned}
E=& {R_2}^2 (1 - 2 {s_1})^2 (-1 + {s_2})^2 + {R_1}^2 (1 - 2 {s_1})^2 {s_2}^2 - 
 2 {R_1} {R_2} (8 (-1 + {s_1})^2 {s_1}^2  \\&+ {s_2} - 
    12 (-1 + {s_1}) {s_1} {s_2} + (7 + 12 (-1 + {s_1}) {s_1}) {s_2}^2 - 16 {s_2}^3 + 
    8 {s_2}^4).
\end{aligned}
\end{equation*} If $E=0$, the system fails to be semitoric. 
\label{th:nffintro}
\end{theo}

Theorem \ref{th:nffintro} is a reformulation of Theorem \ref{th:semit} and Corollary \ref{co:nff} stated later in the paper. The number of focus-focus points invariant is illustrated in Figure \ref{fig:4X2nff}.
The image of the momentum map of system \eqref{sysdefintro} is plotted in Figure \ref{ff:arrayOfPics} which is the starting point for the computation of the polygon invariant.

\begin{theo}
 The polygon invariant is computed in Theorem \ref{th:poly} and plotted in Figure \ref{ff:2FFpolygons}.
\end{theo}

Whenever $\nff=2$, the height invariant is defined and has two components. Their explicit computation is our main result:

\begin{theo}
For the values of $(s_1,s_2)$ for which the system \eqref{sysdefintro} has two focus-focus singularities, the height invariant $h :=(h_1,h_2)$ is given by
\begin{align*}
h_1 &= -\dfrac{1}{2\pi} \mcF(s_1,s_2,R) + 2 u\left( (s_1-\tfrac{1}{2} )(s_2 - \tfrac{R}{R+1}) \right), \\[0.2cm]
h_2&= \hspace{0.3cm} \dfrac{1}{2\pi} \mcF(s_1,s_2,R)  +2 u\left( -(s_1-\tfrac{1}{2} )(s_2 - \tfrac{R}{R+1}) \right) = 2-h_1
\end{align*}
where $R:=\tfrac{R_2}{R_1}$, $u$ is the Heaviside step function,
\begin{align*}
\mcF&(s_1,s_2,R) := 2 R \arctan \left(\frac{\symbga_C}{\sqrt{\symbga_A} (2 {s_1}-1) (R ({s_2}-1)+{s_2})}\right) \\
&+ 2 \arctan\left(\frac{\symbga_D}{\sqrt{\symbga_A} (2 {s_1}-1) (R ({s_2}-1)+{s_2})}\right) \\
       &+ \frac{(2 {s_1}-1) (R ({s_2}-1)+{s_2}) }{2 \left({s_1}^2-{s_1}+{s_2}^2-{s_2}\right)}\log \left(\frac{-\sqrt{\symbga_B}}{2 (R+1) \left({s_1}^2-{s_1}+{s_2}^2-{s_2}\right)+\sqrt{\symbga_A}}\right).
\end{align*}
The height invariant is plotted in Figure \ref{plot:intro} and the coefficients $\symbga_A$ $\symbga_B$, $\symbga_C$ and $\symbga_D$ are explicitly stated in Proposition \ref{prop:coeff} below.
\label{th:intro}
\end{theo}

\begin{figure}[ht]
 \centering
  	\begin{subfigure}[b]{6cm}
        \includegraphics[width=\textwidth]{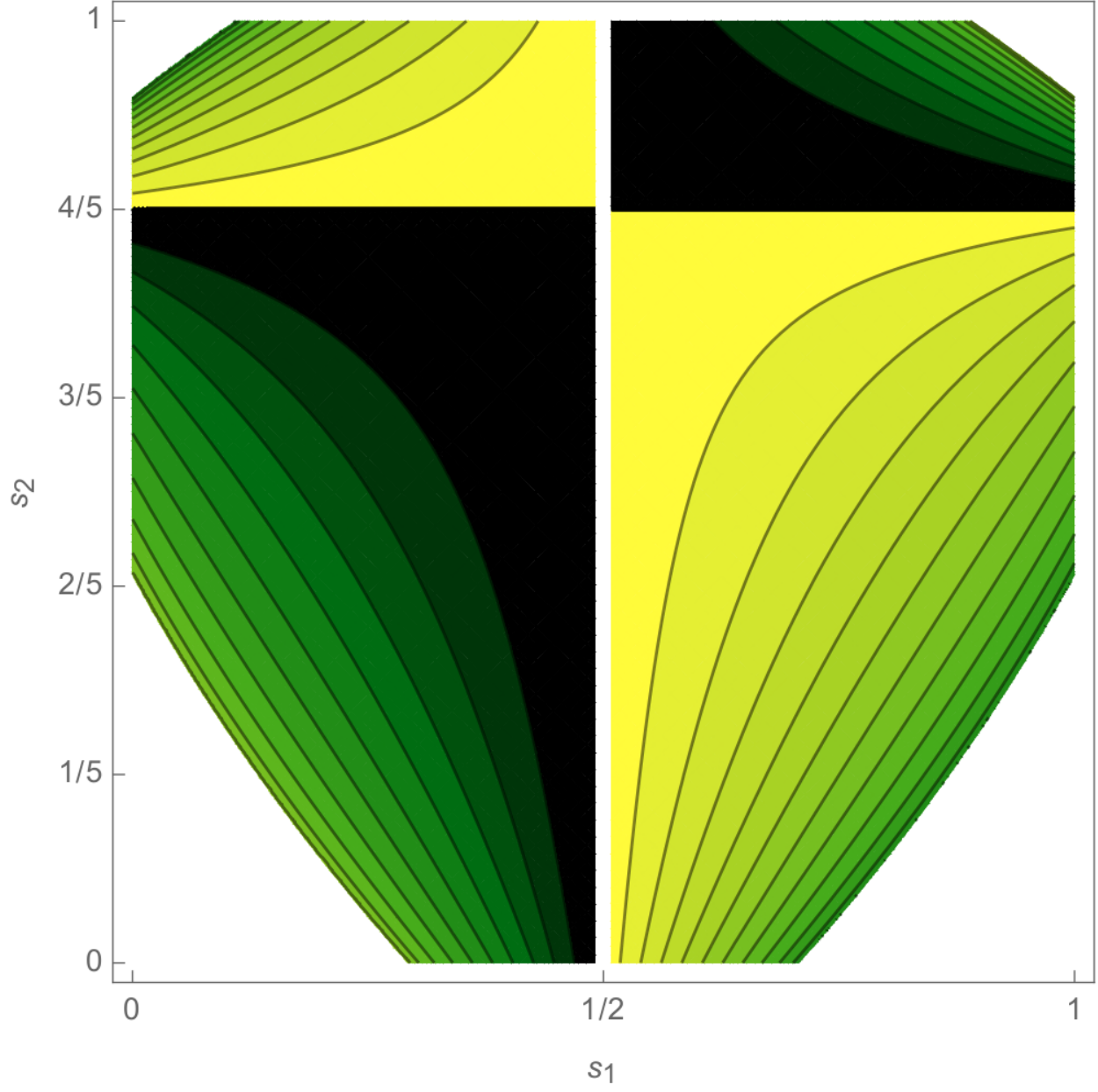}
        \caption{\small \textit{Singularity $N \times S$}}
    \end{subfigure}
    \hspace{1.5cm}
        \begin{subfigure}[b]{6cm}
        \includegraphics[width=\textwidth]{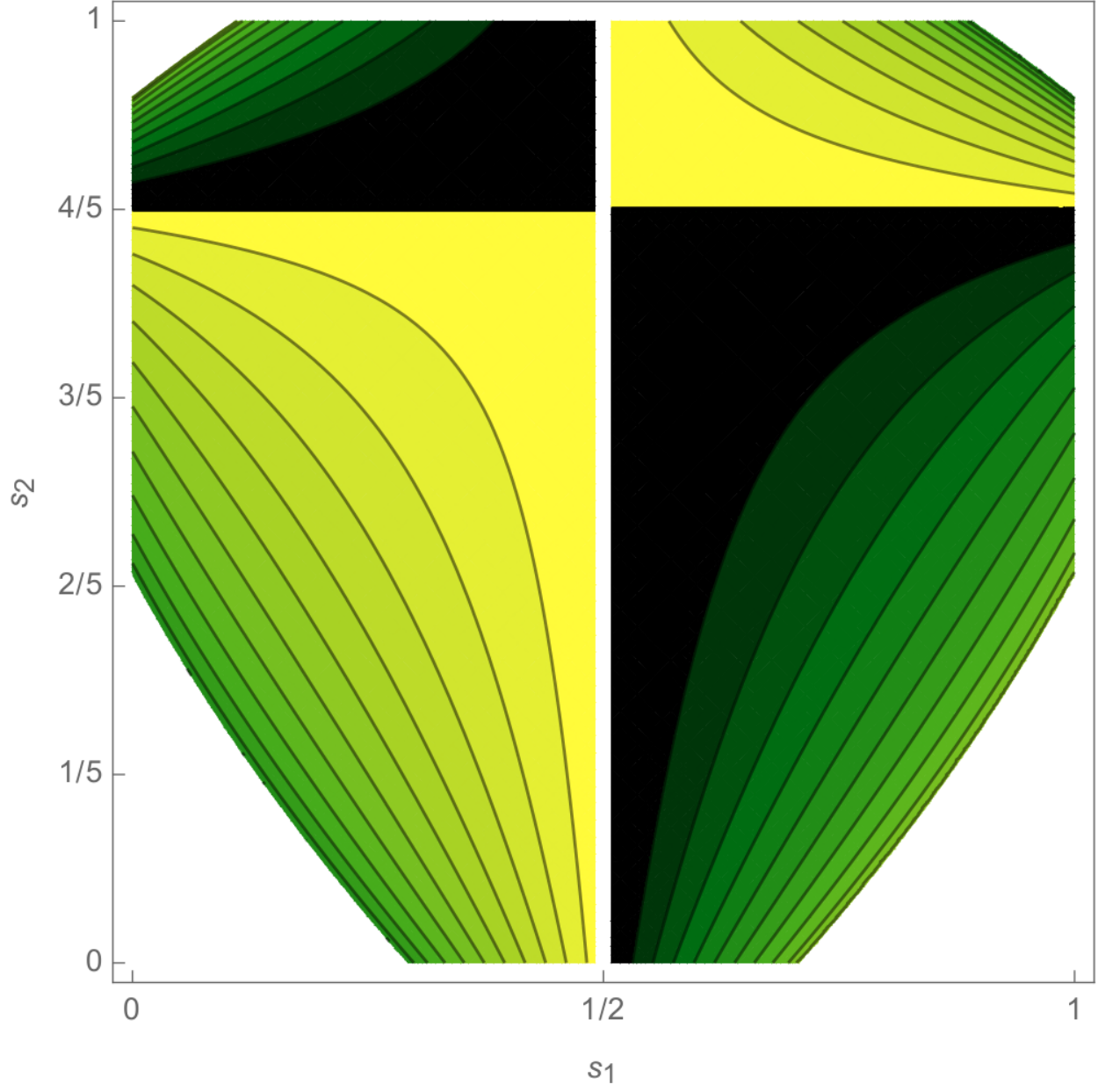}
        \caption{\small \textit{Singularity $S \times N$}}
    \end{subfigure}
\caption{\small \textit{Representation of the height invariant for $R_1=1$, $R_2=4$. Lighter colours represent higher values than darker colours.}}
 \label{plot:intro}
\end{figure}

The coefficients encode the dependence of the height invariant on the various parameters. This dependence is polynomial, except for some radicals. 

\begin{pro}
\label{prop:coeff}
 The coefficients $\symbga_A$ $\symbga_B$, $\symbga_C$ and $\symbga_D$ of Theorem \ref{th:intro} are given by
\begin{align*}
 \symbga_A := & -R^2 (1-2 {s_1})^2 ({s_2}-1)^2+2 R \left(8 {s_1}^4-16 {s_1}^3+4 {s_1}^2 \left(3 {s_2}^2-3 {s_2}+2\right) \right.\\
 &\left.-12 {s_1} ({s_2}-1) {s_2}+{s_2} \left(8 {s_2}^3-16 {s_2}^2+7 {s_2}+1\right)\right)-(1-2 {s_1})^2 {s_2}^2 \\[0.2cm]
 \symbga_B := &  R^2 \left(4 {s_1}^4-8 {s_1}^3+4 {s_1}^2 \left(3 {s_2}^2-4 {s_2}+2\right)-4 {s_1} \left(3 {s_2}^2-4 {s_2}+1\right)\right.\\
 &\left.+({s_2}-1)^2 \left(4 {s_2}^2+1\right)\right)-2 R \left(4 {s_1}^4-8 {s_1}^3+4 {s_1}^2 \left({s_2}^2-{s_2}+1\right) \right.\\
 &\left.-4 {s_1} ({s_2}-1) {s_2}+{s_2} \left(4 {s_2}^3-8 {s_2}^2+3 {s_2}+1\right)\right)+4 {s_1}^4-8 {s_1}^3\\
 &+4 {s_1}^2 \left(3 {s_2}^2-2 {s_2}+1\right)+4 {s_1} {s_2} (2-3 {s_2})+{s_2}^2 \left(4 {s_2}^2-8 {s_2}+5\right) \\[0.2cm]
 \symbga_C := & -4 R^2 {s_1}^2 {s_2}^2+8 R^2 {s_1}^2 {s_2}-4 R^2 {s_1}^2+4 R^2 {s_1} {s_2}^2-8 R^2 {s_1} {s_2}+4 R^2 {s_1}\\
 &-R^2 {s_2}^2+2 R^2 {s_2}-R^2+8 R {s_1}^4-16 R {s_1}^3+8 R {s_1}^2 {s_2}^2-8 R {s_1}^2 {s_2}\\&+8 R {s_1}^2-8 R {s_1} {s_2}^2+8 R {s_1} {s_2}+8 R {s_2}^4-16 R {s_2}^3+6 R {s_2}^2+2 R {s_2}\\
 &+4 \sqrt{{\symbga_B}} \left(-{s_1}^2+{s_1}-{s_2}^2+{s_2}\right)-8 {s_1}^4+16 {s_1}^3-20 {s_1}^2 {s_2}^2+16 {s_1}^2 {s_2}\\
 &-8 {s_1}^2+20 {s_1} {s_2}^2-16 {s_1} {s_2}-8 {s_2}^4+16 {s_2}^3-9 {s_2}^2 \\[0.2cm]
 \symbga_D  :=& -8 R^2 {s_1}^4+16 R^2 {s_1}^3-20 R^2 {s_1}^2 {s_2}^2+24 R^2 {s_1}^2 {s_2}-12 R^2 {s_1}^2+20 R^2 {s_1} {s_2}^2\\&-24 R^2 {s_1} {s_2}+4 R^2 {s_1}-8 R^2 {s_2}^4+16 R^2 {s_2}^3-9 R^2 {s_2}^2+2 R^2 {s_2}-R^2\\
 &+4 R \sqrt{{\symbga_B}} \left(-{s_1}^2+{s_1}-{s_2}^2+{s_2}\right)+8 R {s_1}^4-16 R {s_1}^3+8 R {s_1}^2 {s_2}^2-8 R {s_1}^2 {s_2}\\
 &+8 R {s_1}^2-8 R {s_1} {s_2}^2+8 R {s_1} {s_2}+8 R {s_2}^4-16 R {s_2}^3+6 R {s_2}^2+2 R {s_2}-4 {s_1}^2 {s_2}^2\\
 &+4 {s_1} {s_2}^2-{s_2}^2.
\end{align*}
\end{pro}

Theorem \ref{th:intro} implies the following relation between the height invariant and the parameters of system \eqref{sysdefintro}:

\begin{co}
\label{co:intro}
The two components $(h_1, h_2)$ of the height invariant have an intricate dependence on the four parameters $s_1, s_1, R_1, R_2$ of the system but a very simple relation between each other, namely $h_2=2-h_1$.
\end{co}

Theorem \ref{th:intro} and Proposition \ref{prop:coeff} are restated and proven as Theorem \ref{th:height} later in the paper. Corollary \ref{co:intro} reappears as Corollary \ref{co:rels} at the very end of the paper.

All computations in this paper were verified with \emph{Mathematica}.

\subsubsection*{Structure of the paper}

This paper is structured as follows. In section 2, we briefly summarise the definition of simple semitoric systems and the classification in terms of symplectic invariants. In section 3, we introduce our family of semitoric systems and compute the number of focus-focus points and the polygon invariant. Section 4 is devoted to the computation of the height invariant associated to both focus-focus singularities. 

\subsubsection*{Figures}

All figures have been made with \emph{Mathematica}. Figure \ref{figfib2} has also been edited with \emph{Inkscape}.

%%%%%%%%%%%%%%%%%%%%%%%%%%%%%%%%%%%%%%%%%%%%%%%%%%%%%%%%%%%%%%%%%%%%%%%%%
%%%%%%%%%%%%%%  new section  %%%%%%%%%%%%%%%%%%%%%%%%%%%%%%%%%%%%%%%%%%%%
%%%%%%%%%%%%%%%%%%%%%%%%%%%%%%%%%%%%%%%%%%%%%%%%%%%%%%%%%%%%%%%%%%%%%%%%%%%

\section{The symplectic invariants of semitoric systems}
\label{ss:sinv}

In this section we briefly summarise the symplectic classification of simple semitoric systems. More details can be found in the original papers by Pelayo $\&$ \vungoc \cite{PV1, PV4}.

Let $(M,\om)$ denote a connected four-dimensional  symplectic manifold. The triplet $(M,\om,F)$ is said to be a \emph{completely integrable system} if $F:=(L,H):M \to \R^2$ is a smooth map such that $dF$ has maximal rank almost everywhere and the components $L,H$ Poisson-commute, i.e.\ $\{L,H\}:= \om(\mcX_L,\mcX_H)=0$, where $\mcX_f$ is the Hamiltonian vector field associated to a smooth map $f:M \to \R$ via $\om(\mcX_f,\cdot) = -\dee f$. This definition of integrability is sometimes referred to as \emph{Liouville integrability} in the literature. The singularities of $(M,\om,F)$ are the points where the differentials $DL$, $DH$ fail to be linearly independent, so the rank of $DF$ is not maximal. Non-degenerate singularities (see Bolsinov $\&$ Fomenko \cite{BF} or Vey \cite{Ve} for a precise definition) can be locally characterised using normal forms, cf.\ Eliasson \cite{El1, El2}, Miranda $\&$ Zung \cite{MZ}, Miranda $\&$ \vungoc \cite{MV}, \vungoc $\&$ Wacheaux \cite{VW}, and others. In particular, they can be decomposed into \emph{regular}, \emph{elliptic}, \emph{hyperbolic} and \emph{focus-focus} components.

\begin{de}
A \emph{semitoric system} is a completely integrable system $(M,\om,F)$ with two degrees of freedom, where $F:=(L,H): M \to \R^2$ is smooth and the following conditions are satisfied:
\begin{enumerate}[\quad 1)]
	\item All singularities are non-degenerate and have no hyperbolic components. 
	\item The map $L$ induces an effective $\mbS^1$-action on $M$ with a $2\pi$-periodic flow.
	\item $L$ is proper, i.e.\ the preimage of a compact set by $L$ is compact again.
\end{enumerate} Moreover, if the following condition is satisfied, the semitoric system is said to be \emph{simple}:
\begin{enumerate}[\quad 1)]
\setcounter{enumi}{3}
	\item In each level set of $L$ there is at most one singularity of focus-focus type.
\end{enumerate}
\label{def:semitoric}
\end{de} 

In the present work we will only consider simple semitoric systems. In the context of semitoric systems, we have two degrees of freedom and we exclude hyperbolic components, so the rank of $DF$ can only be 0 or 1. Singularities of rank 0 can either be of \emph{focus-focus} type or of \emph{elliptic-elliptic} type, i.e.\ having two elliptic components. Singularities of rank 1 must necessarily have a regular and an elliptic component, so they are called \emph{elliptic-regular} singularities.

\subsection{The singular Lagrangian fibration}

The momentum map $F:=(L,H)$ of a semitoric system induces a two-dimensional singular \emph{Lagrangian} fibration on $M$. The base of this fibration is the image $B:=F(M)$, which is a contractible subset of $\R^2$. Since $L$ is proper, we know that all fibres of $F$ are compact. \vungoc \cite{Vu2} showed that the fibres of semitoric systems are connected and he also characterised the structure of the fibration over $B$. The boundary $\partial B \subset B$ consists of all elliptic-elliptic and elliptic-regular critical values. The former are located in the vertices, while the latter always come in one-parameter families and form the edges. The set $B_{FF}$ of focus-focus critical values is finite and lies in the interior of $B$. The regular fibres are thus mapped to $B_{\text{reg}} := \mathring{B} \backslash B_{FF}$. %Figure \ref{figfib} shows the different possible fibres that simple semitoric systems can have.

Singular fibres are those containing a singularity.  Elliptic-elliptic singularities constitute always their own fibre. Elliptic-regular fibres are homeomorphic to a circle. Fibres containing a focus-focus singularity are homeomorphic to a pinched torus, see Figure \ref{figfib2}. If simplicity is not assumed (cf.\ Definition \ref{def:semitoric}), then fibres containing more than one focus-focus singularity, homeomorphic to a multi-pinched torus, are also possible. This situation has been studied by Pelayo $\&$ Tang \cite{PT} and Palmer \emph{et al.}\ \cite{PPT}.

Regular fibres are those containing no singularities. According to the action-angle theorem by Liouville $\&$ Arnold $\&$ Mineur \cite{Ar}, regular fibres are homeomorphic to the two-torus $\T^2$. More precisely, for each regular value $c \in B_\text{reg}$ we can find a neighbourhood $U$ of $c$ and $V \subset \R^2$ of the origin such that $F^{-1}(U) \subset M$ is symplectically equivalent to $V \times \T^2 \subset T^* \T^2$. This defines an integral-affine structure on $B_\text{reg}$. Let $\Lam_c = F^{-1}(c) \simeq \T^2$ be the fibre corresponding to the value $c$ and let $\{\gamma_1(c), \gamma_2(c)\}$ be a basis of $H^1(\Lam_c)$, varying smoothly with $c$. Then the action coordinates $(I_1,I_2)$ on $V$ are given by the expression
\begin{equation}
I_j(c):=\dfrac{1}{2\pi} \oint_{\gamma_j(c)} \varpi,\qquad j=1,2,
\label{actions}
\end{equation}where $\varpi$ is any semiglobal primitive of the symplectic form, $\dee \varpi = \om$.

Since the Hamiltonian flow of $L$ induces a global circular action on $M$, we can take $\gamma_1(c)$ to be the orbit of $L$. This way, we will have $I_1(c) = L(c)$. Different choices of $\gamma_2(c)$ belonging to different homology classes will result in different values of $I_2(c)$, so there is an integer degree of freedom in the definition of this action coordinate.

\begin{figure}[ht]
\centering
\includegraphics[height=7cm]{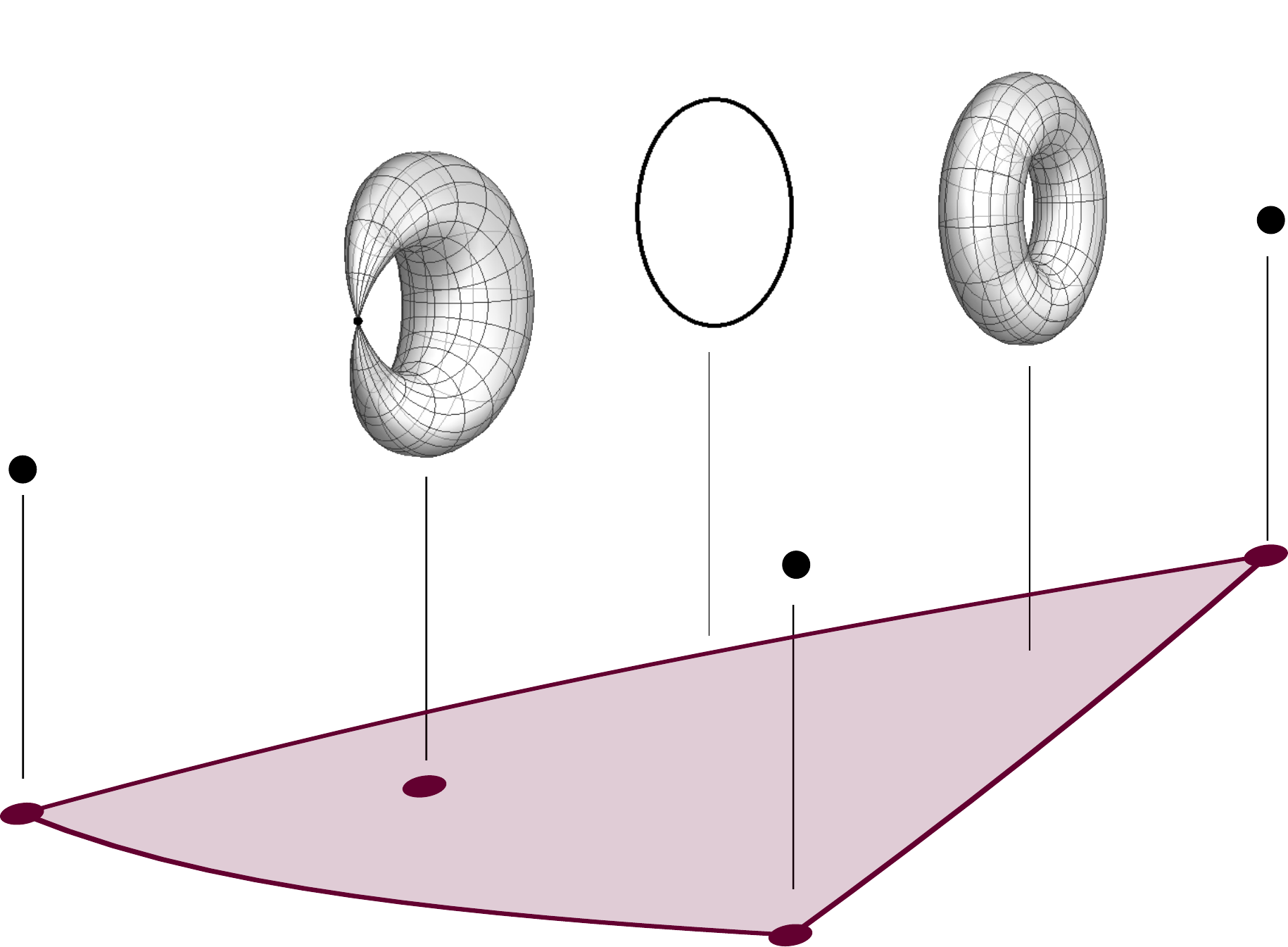}
\caption{\textit{\small{Example fibration of a semitoric system, corresponding to the coupled angular momenta for $t=1/2$ and $R_2>R_1$. The fibration has three elliptic-elliptic fibres, one focus-focus fibre and three 1-parameter families of elliptic-regular fibres. On $B_\text{reg}$ the fibres are regular 2-tori.}}}
\label{figfib2}
\end{figure}

\subsection{The polygon and height invariants}

\vungoc \cite{Vu2} used the action coordinates to define the so-called \emph{cartographic homeomorphism} as follows. Let $\nff \in \Z$ be the number of focus-focus points and $c_1,...,c_\nff \in B$ their critical values. For each $r =1,...,\nff$, pick a sign choice $\epsilon_r \in \{-1,1\} \simeq \Z_2$  and consider the half-line $b_r^{\epsilon_r} \subset B$ that starts in $c_r$ and extends upwards if $\epsilon_r = +1$ and downwards if $\epsilon_r = -1$. Let $b^\epsilon = \cup_r b_r^{\epsilon_r}$.  Then for any set of choices $\epsilon = (\epsilon_1,...,\epsilon_\nff)$ there exists a map $f:=f^{\epsilon}:B \to \R^2$ that is a homeomorphism onto its image $\Delta := f(M)$, it preserves the first coordinate, i.e.\ $f(l,h) = (l,f^{(2)}(l,h))$ and $f|_{B\backslash b^\epsilon}$ is a diffeomorphism onto its image. This process is illustrated in Figure \ref{fig:straight}.

\begin{figure}[ht]
\centering
  	\begin{subfigure}[b]{6cm}
  	\centering
        \includegraphics[width=0.8\textwidth]{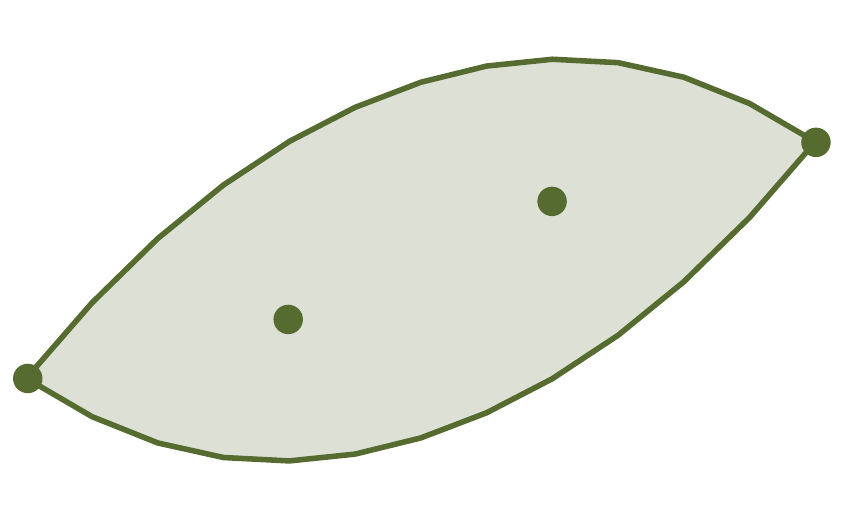}
        \caption{\small \textit{Image $B=F(M)$}}
    \end{subfigure}
    \begin{subfigure}[b]{1.5cm}
    \centering
    \vspace{-2cm}$\Longrightarrow$ \vspace{2cm}
    \end{subfigure}
        \begin{subfigure}[b]{6.5cm}
        \includegraphics[width=\textwidth]{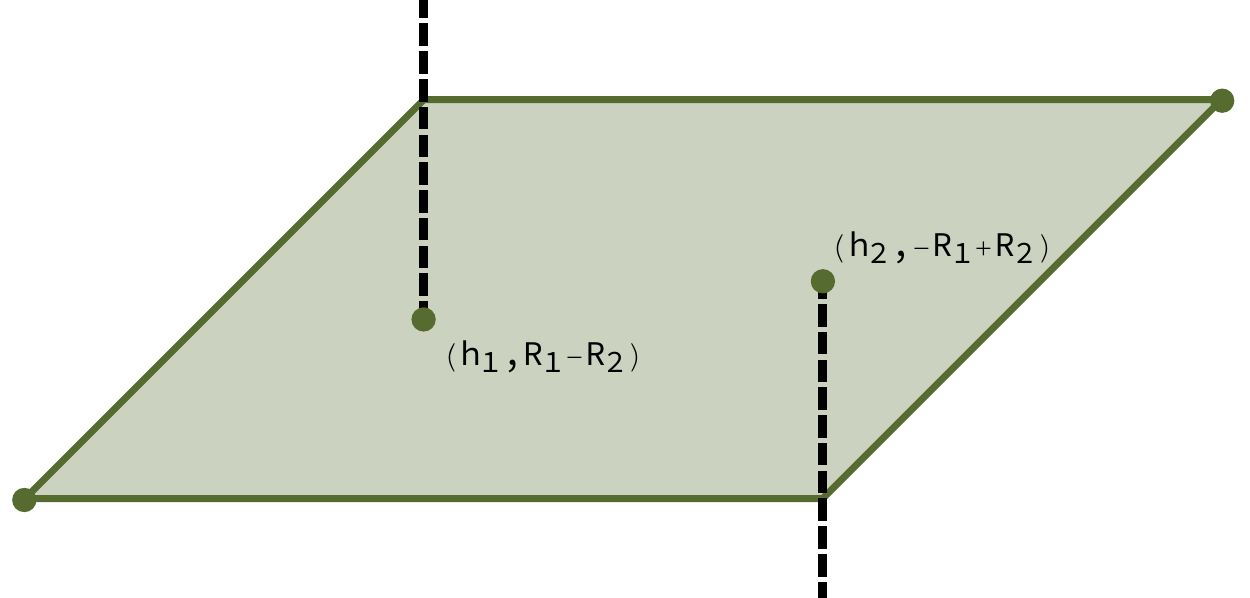}
        \caption{\small \textit{Polygon $\De = f(B)$}}
    \end{subfigure}

\caption{\textit{\small{The cartographic homeomorphism $f$ brings the image of the momentum map $F(M)=:B$ to a polygon $f(B)=:\Delta $. It preserves the first coordinate and adds corners at the focus-focus values following the cutting directions $\epsilon = (\epsilon_1, \epsilon_2)$. In this case, we have $\epsilon = (1,-1)$.}}}
\label{fig:straight}
\end{figure}

The map $f$ is constructed by extending the coordinates $(L,I_2)$ defined by equation \eqref{actions}. The non-smoothness along the segments $b^\epsilon_r$ is a consequence of the monodromy induced by the presence of focus-focus singularities, an obstruction to globally-defined action-angle coordinates studied, among others, by Nekhoroshev \cite{Ne} and Duistermaat \cite{Du1}. 

The image $\Delta \subset \R^2$ of the cartographic homeomorphism is a convex rational polygon, which is compact if and only if $M$ is compact. Since the definition of the action $I_2$ is not unique, neither is $\Delta$. There is a $\Z$-action that relates all possible choices of $I_2$. Besides that, there is a $(\Z_2)^\nff$-action of sign choices $\epsilon$ that also acts on $f$.

\begin{de}
The \emph{polygon invariant} associated with the simple semitoric system $(M,\om,F)$ is the equivalence class $[\De]$ of the polygon $\Delta = f(B)$ by the $(\Z_2)^\nff \times \Z$-action that relates all possible choices of $f$.
\label{df:polinv}
\end{de}

The image of the focus-focus critical values under the cartographic homeomorphism is the intuition for next invariant: for each focus-focus critical value $c_r$, $r=1,...,\nff$, we can compute the vertical distance between the image of the critical value under $f$ and the edge of the polygon $\Delta$,
\begin{equation}
	h_r := f^{(2)}(c_r) - \min_{c \in \Delta \cap b^{-1}_r} \pi_2(c),
\end{equation} where $\pi_2:\R^2 \to \R$ is the canonical projection onto the second coordinate. This quantity is independent of the choice of map $f$. The  height $h_r$ can also be interpreted as the symplectic volume of the submanifold $Y_r^- := \{ p \in M$ $|$ $L(p) = L(m_r)$ and $H(p) < H(m_r)\}$, that is, the real volume of $Y_r^-$ divided by $2\pi$, where $m_r \in M$ is the focus-focus singularity corresponding to $c_r$.  

\begin{de}
The \emph{height invariant} associated to the simple semitoric system $(M,\om,F)$ is the $\nff$-tuple $h = (h_1,...,h_\nff)$. It is independent of the choice of cartographic homeomorphism $f$.
\label{df:heiinv}
\end{de}

\subsection{The other invariants and the symplectic classification}

The remaining two invariants are related to the structure of the action $I_2$ as we approach focus-focus singularities. Fix $r \in \{1,...,\nff\}$. In this paper, we do not work with these two invariants, but for sake of completeness, we review them quickly. More details can be found in Pelayo $\&$ \vungoc \cite{PV1, PV4}. In a neighbourhood of the singular fibre containing $m_r$, \vungoc \cite{Vu1} proved that the action $I_2$ can be written as
\begin{equation*}
2\pi I_2(w) = 2\pi I_2(0) - \text{Im}(w\log w -w) + S_r(w),
\end{equation*} where $w:= l + ij$, $i$ is the imaginary unit, $l$ is the value of $L-L(m_r)$, and $j$ is the value of the second Eliasson function around $m_r$. The function $S_r(w)$ is a smooth function that can be understood as a desingularised action. Different choices of $I_2$ change $S_r$ by a multiple of $2\pi l$, so we can fix a choice $I_{2,r}$ of $I_2$ in this neighbourhood by imposing $0 \leq \partial_l S_r(0) < 2\pi$. If we denote its Taylor series by $S_r^\infty$, then the $\nff$-tuple $S^\infty = (S_1^\infty, ..., S_\nff^\infty)$ is the \emph{Taylor series invariant}. In \S \ref{ss:comphei} we show that $h_r$ can also be related to $I_{2,r}(0)$.

Fix now a polygon $\De$ and its corresponding homeomorphism $f$. Then, for each $r=1,...,\nff$, the values of $f^{(2)}$ around $c_r$ will differ from those of $I_{2,r}$ by a multiple $\kappa_r \in \Z$ of $2\pi l$. The $\nff$-tuple $\kappa = (\kappa_1,...,\kappa_\nff)$ depends on the choice of $f$. The equivalence class of $\kappa$ under the $(\Z_2)^\nff \times \Z$-action that acts on $f$ determines the \emph{twisting index invariant}.

Consider now the following definition of isomorphism between semitoric systems:

\begin{de}
 Two semitoric systems $(M_1,\om_1,F_1)$ and $(M_2,\om_2,F_2)$ are said to be isomorphic if there exists a pair $(\varphi, \varrho)$, such that $\varphi:(M_1,\om_1) \to (M_2, \om_2)$ is a symplectomorphism and $\varrho: B_1 \to B_2$ is a diffeomorphism between $B_1 := F_1(M_1)$ and $B_2 := F_2(M_2)$ that satisfies $\varrho \circ F_1 = F_2 \circ \varphi$ and is of the form 
 $$\varrho(l,h) = (l,\varrho^{(2)}(l,h)),\quad \dfrac{\partial \varrho^{(2)}}{\partial h} >0.$$ The pair $(\varphi, \varrho)$ is called \emph{semitoric isomorphism}. 
\end{de}

Pelayo $\&$ \vungoc \cite{PV1, PV4} give a classification of simple semitoric systems up to isomorphism using the number of focus-focus points $\nff$ and the other four invariants introduced in this section.

\begin{theo}[Pelayo $\&$ \vungoc \cite{PV1, PV4}] There exists a symplectic classifications of simple semitoric systems in the following sense:
\begin{enumerate}[1)]
 \item 
 To each simple semitoric system we can associate the following five symplectic invariants, namely the number of focus-focus points $\nff$, the polygon invariant $[\Delta]$, the height invariant $h$, the Taylor series invariant $S^\infty$, and the twisting index invariant $[\kappa]$.
 \item
 Two semitoric systems are isomorphic if and only if their list of invariants coincide.
 \item
 Given a list of admissible invariants, there exists a simple semitoric system $(M, \om, F)$ that has that list as its list of invariants.
\end{enumerate}
\label{th:class}
\end{theo}

%%%%%%%%%%%%%%%%%%%%%%%%%%%%%%%%%%%%%%%%%%%%%%%%%%%%%%%%%%%%%%%%%%%%%%%%%
%%%%%%%%%%%%%%  new section  %%%%%%%%%%%%%%%%%%%%%%%%%%%%%%%%%%%%%%%%%%%%
%%%%%%%%%%%%%%%%%%%%%%%%%%%%%%%%%%%%%%%%%%%%%%%%%%%%%%%%%%%%%%%%%%%%%%%%%%%

\section{A symmetric family with two focus-focus points}

Consider $M=\mbS^2 \times \mbS^2$, together with the symplectic form $\omega = - (R_1\, \omega_{\mbS^2}  \oplus  R_2\, \omega_{\mbS^2})$, where $\omega_{\mbS^2}$ is the standard symplectic form of the unit sphere $\mbS^2$ and $R_1,R_2$ are two positive real numbers. Consider Cartesian coordinates $(x_1,y_1,z_1, x_2,y_2,z_2)$ on $M$, where $(x_i,y_i,z_i),$ $i=1,2$, are Cartesian coordinates on the unit sphere $\mbS^2 \subset \R^3$. We define a 4-parameter family of integrable systems $(M,\om,(L,H))$, where $L,H:M \to \mathbb{R}$ are the smooth functions given by
\begin{equation}
\begin{cases}
L(x_1,y_1,z_1,x_2,y_2,z_2)\;:=R_1 z_1  + R_2 z_2, \\
H(x_1,y_1,z_1,x_2,y_2,z_2):=(1-2s_1)(1-s_2) z_1 + (1-2s_1)s_2 z_2 
\\ \hspace{4.3cm}+ 2(s_1+s_2-{s_1}^2 - {s_2}^2) (x_1x_2 + y_1y_2).\\
\end{cases}
\label{sysdef}
\end{equation}

The parameters $R_1,R_2$ are called \emph{geometric parameters}, because they are related to the symplectic manifold. The parameters $s_1,s_2 \in [0,1]$ are the \emph{coupling parameters} of the system. For now, we will assume that $R_2 > R_1$. The function $L$ represents the sum of the height functions on both spheres and its Hamiltonian vector field corresponds to a simultaneous rotation of both spheres around the vertical axis. The function $H$ corresponds to an interpolation among rotations around the vertical axis on the first sphere, the second sphere and the relative polar angle between the two position vectors, see Figure \ref{sysphase}. 

\begin{figure}[ht]
\centering
\includegraphics[width=12cm]{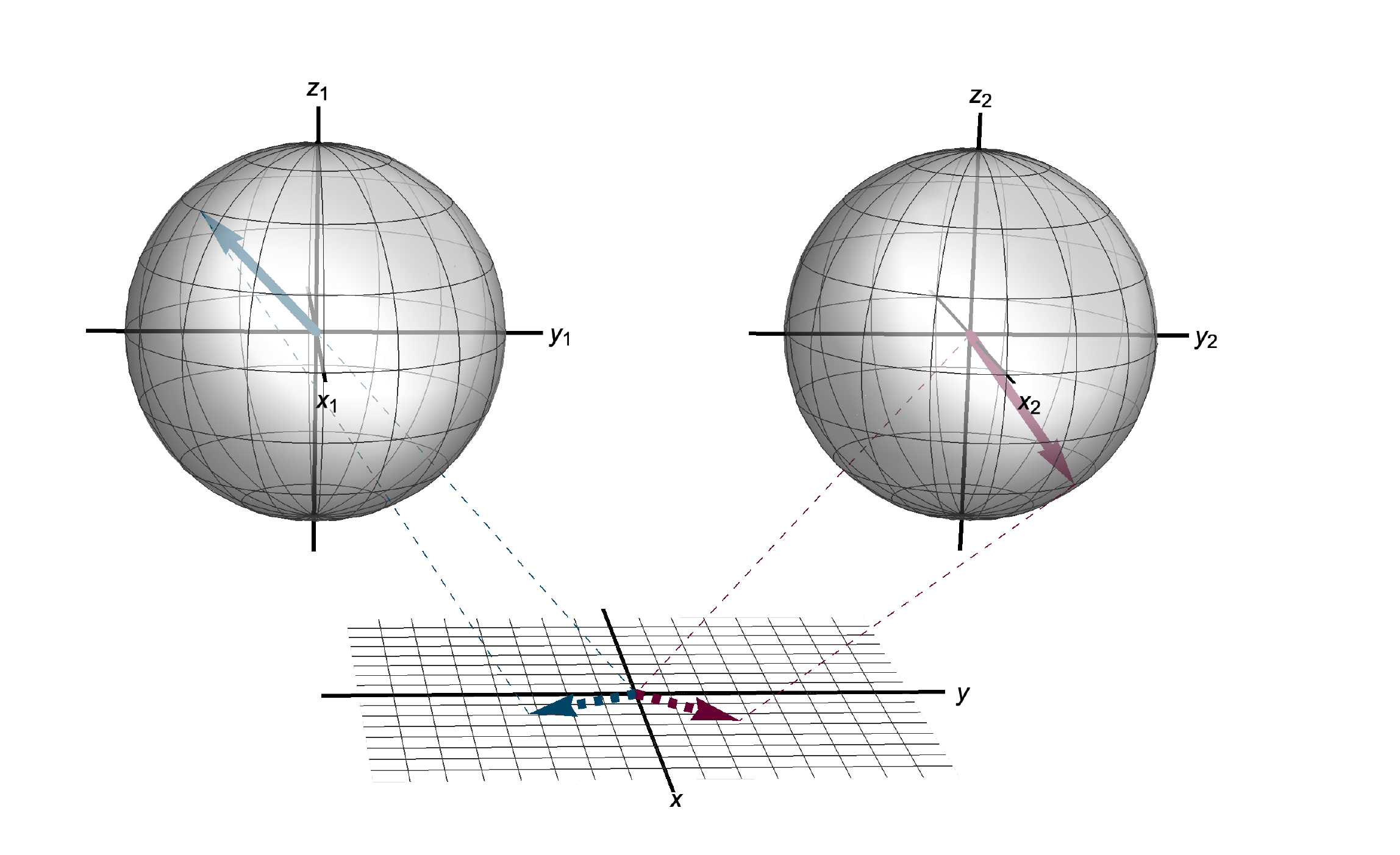}
\caption{\small{\textit{Representation of the system \eqref{sysdef}. The phase space is $M=\mbS^2 \times \mbS^2$, with position vectors $(x_1,y_1,z_1)$ in blue and $(x_2,y_2,z_2)$ in red. These vectors determine the value of the functions $L,H$.}}}
\label{sysphase}
\end{figure}

\begin{re}
The system \eqref{sysdef} is a particular case of the general family of systems
\begin{equation*}
\begin{cases}
L(x_1,y_1,z_1,x_2,y_2,z_2)\;:=R_1 z_1  + R_2 z_2, \\
H(x_1,y_1,z_1,x_2,y_2,z_2):=t_1 z_1 + t_2 z_2 + t_3 (x_1x_2 + y_1y_2) + t_4 z_1 z_2
\end{cases}
\end{equation*}
defined in Hohloch $\&$ Palmer \cite{HP} obtained by setting
$$ \begin{aligned}
&t_1 = (1-2s_1)(1-s_2) &\hspace{1.5cm}& t_3 = 2(s_1+s_2-{s_1}^2 - {s_2}^2) \\
&t_2 = (1-2s_1)s_2  && t_4 =0.
\end{aligned} $$
\label{Tparams}
\end{re}
\begin{pro}
The system $(M,\om,(L,H))$ defined by equation \eqref{sysdef} is completely integrable for all choices of radii $0<R_1<R_2$ and coupling parameters $s_1,s_2 \in [0,1]$.
\label{p:integrab}
\end{pro}
\begin{proof}
Remark \ref{Tparams} allows us to use  \cite[Theorem 3.1]{HP}, which states that the system is completely integrable if $t_3 \neq 0$. We have that $t_3 = 2(s_1+s_2-{s_1}^2 - {s_2}^2)>0$ for all values $s_1,s_2 \in [0,1]$ except for the points $(s_1,s_2) \in \{0,1\}^2$. We investigate these four particular cases separately:
\begin{multicols}{2}
\underline{Case $(s_1,s_2)=(0,0)$:}
$$ \begin{cases}
L\;=R_1 z_1  + R_2 z_2, \\
H=z_1, \\
\end{cases}$$ 

\underline{Case $(s_1,s_2)=(0,1)$:}
$$ \begin{cases}
L\;=R_1 z_1  + R_2 z_2, \\
H=z_2,
\end{cases}$$ 
\underline{Case $(s_1,s_2)=(1,0)$:}
$$ \begin{cases}
L\;=R_1 z_1  + R_2 z_2, \\
H=-z_1, \\
\end{cases}$$ 

\underline{Case $(s_1,s_2)=(1,1)$:}
$$ \begin{cases}
L\;=R_1 z_1  + R_2 z_2, \\
H=-z_2.
\end{cases}$$ 
\end{multicols} Since $\{z_i,z_j\}=0$ for $i,j=1,2$, the functions $L$ and $H$ Poisson-commute: $\{L,H\}=0$. Moreover, since $R_1,R_2>0$, $DL$ is linearly independent of $\pm D z_i$ for $i=1,2$, so these four systems are also completely integrable. 
\end{proof}

The four extreme cases considered in the proof of Proposition \ref{p:integrab} are actually of \emph{toric type}, that is, toric up to a diffeomorphism on the base, cf.\ \vungoc \cite[Definition 2.1]{Vu2}. In particular, all their flows are periodic. This is because the flow of $H = \pm z_i$, $i=1,2$ corresponds to rotations around the vertical axis in the $i$-th sphere.

In Figure \ref{ff:arrayOfPics} we can see the evolution of the image of the momentum map $(L,H)$ as we move the coupling parameters $s_1,s_2$. The extreme cases correspond to the images on the four corners. 

\begin{figure}[ht]
\centering
\includegraphics[height=0.5\textheight]{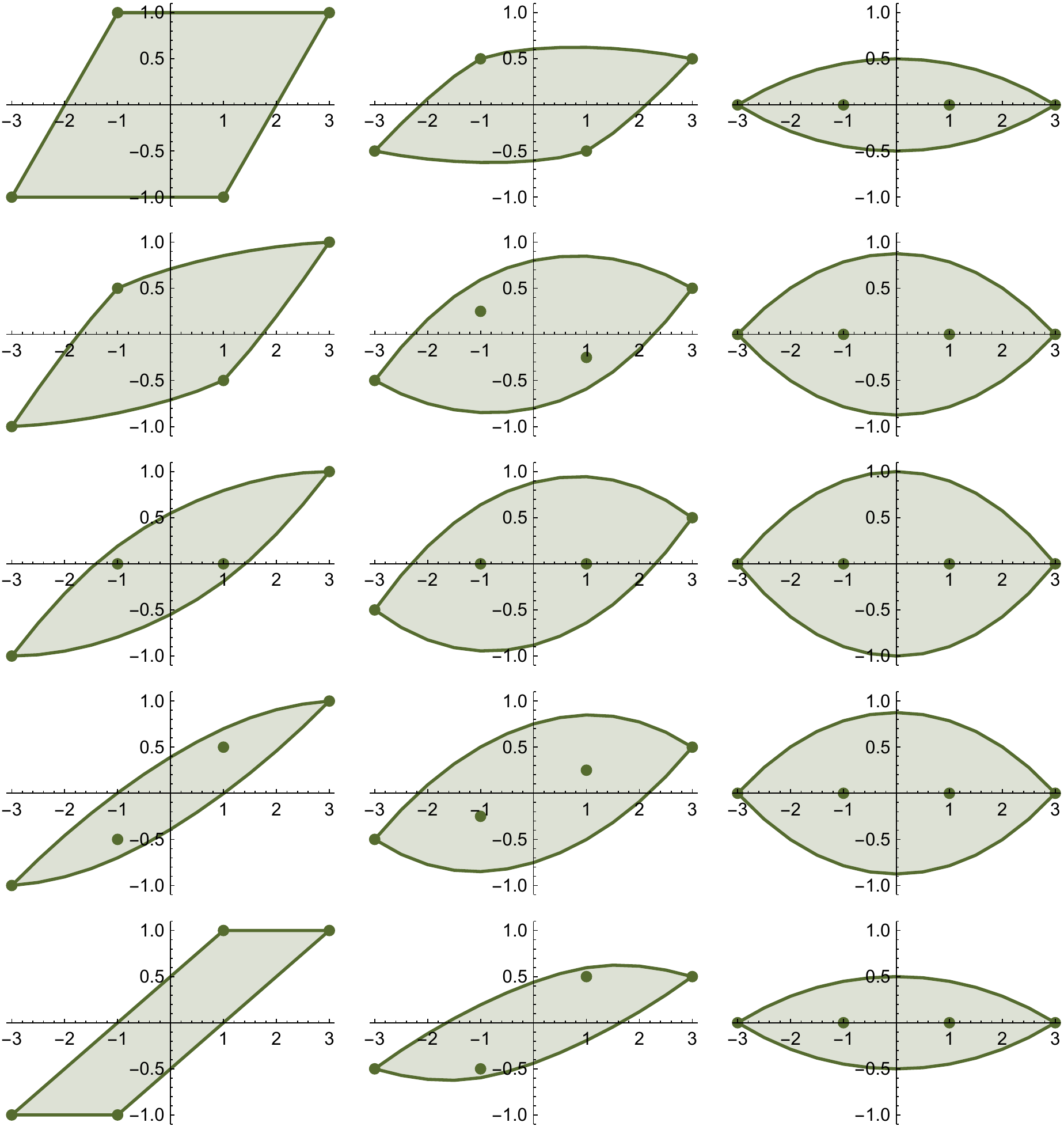}
\includegraphics[height=0.5\textheight]{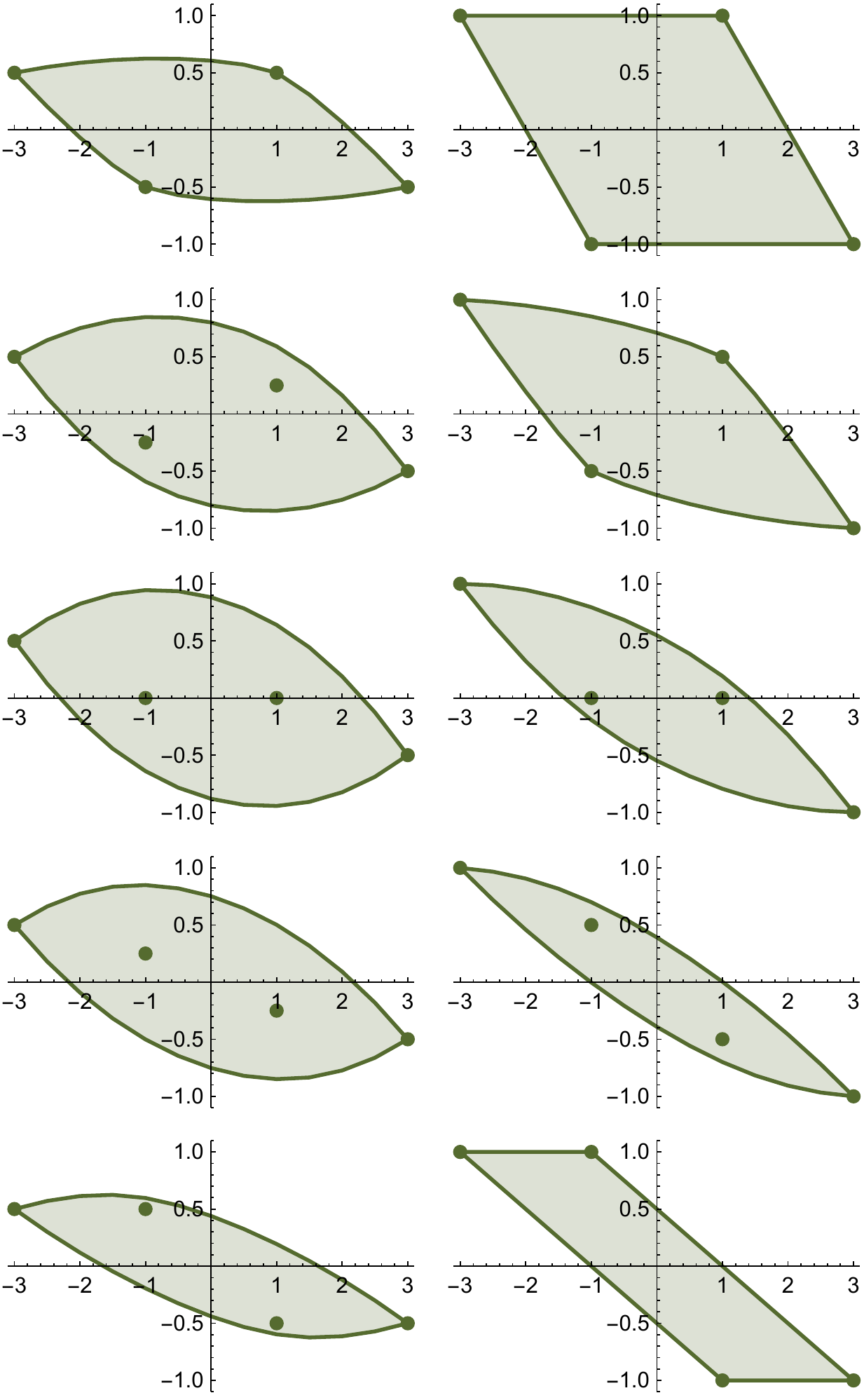}
\caption{\small{\textit{Image of the momentum map $F=(L,H)$ with $R_1=1$, $R_2=2$ for different values of $s_1,s_2 \in [0,1]^2$. The parameter $s_1$ varies horizontally from $0$ (left) to $1$ (right) and $s_2$ varies vertically, from $0$ (top) to $1$ (bottom).}}}
\label{ff:arrayOfPics}
\end{figure} 

\subsection{The number of focus-focus points}
\label{s:4X2nff}

The first symplectic invariant that we compute is the number of focus-focus points. 
\begin{pro}
The rank 0 fixed points of the system \eqref{sysdef} are the four products of poles: $N \times N, N \times S, S \times N$ and $S \times S$. The points $N\times N$ and $S \times S$ are always of elliptic-elliptic type. The points $N\times S$ and $S \times N$ are of elliptic-elliptic type if $E>0$ and of focus-focus type if $E<0$, where
\begin{equation}
\begin{aligned}
E=& {R_2}^2 (1 - 2 {s_1})^2 (-1 + {s_2})^2 + {R_1}^2 (1 - 2 {s_1})^2 {s_2}^2 - 
 2 {R_1} {R_2} (8 (-1 + {s_1})^2 {s_1}^2  \\&+ {s_2} - 
    12 (-1 + {s_1}) {s_1} {s_2} + (7 + 12 (-1 + {s_1}) {s_1}) {s_2}^2 - 16 {s_2}^3 + 
    8 {s_2}^4).
\end{aligned}
\label{eq:E}
\end{equation} The number of focus-focus singularities as a function of the system parameters $s_1,s_2$ is illustrated in Figure \ref{fig:4X2nff}.
\label{pro:nff}
\end{pro}

\begin{figure}[ht]
\centering
\includegraphics[width=7cm]{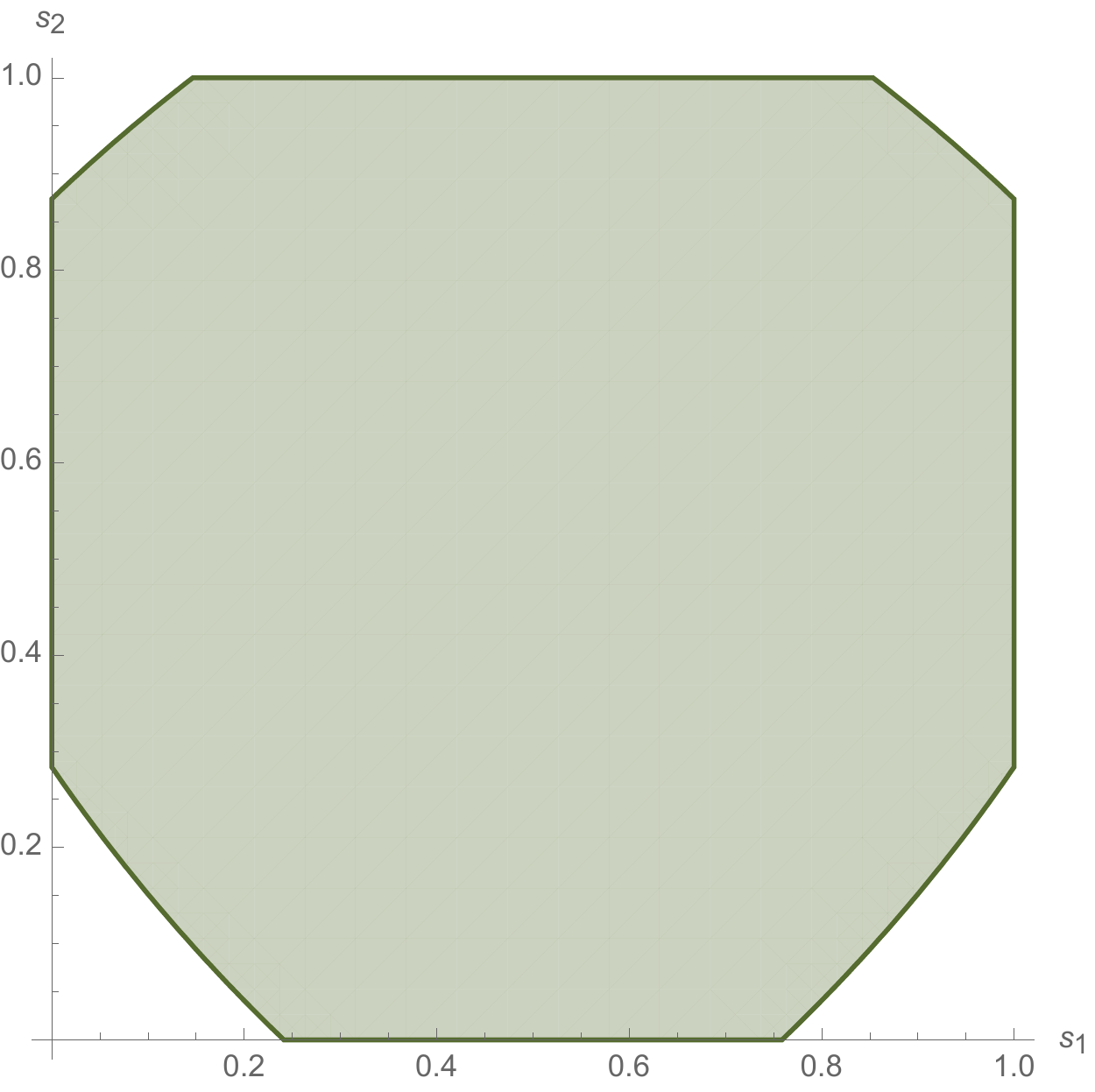}
\caption{\small{\textit{Representation of the number of focus-focus points for $R_1=1$, $R_2=2$ as a function of the coupling parameters $s_1,s_2$. In the green area, we have $\nff=2$ and, in the white one, we have $\nff=0$.}}}
\label{fig:4X2nff}
\end{figure}

\begin{proof}
From \cite[Lemma 3.4]{HP} we know that the product of poles, $N \times N$, $S \times S$, $N \times S$ and $S \times N$, are precisely the rank 0 singularities of the system. From \cite[Corollary 3.6]{HP} we know that the characteristic polynomial of $A_H = \Om^{-1} D^2H$ is
\begin{align*}
\chi(X) =& X^4 + \dfrac{1}{{R_1}^2 {R_2}^2} \left({R_1}^2 ({t_2} + {z_1} {t_4})^2 + 2 {z_1} {z_2} {R_1} {R_2} {t_3}^2 + 
    {R_2}^2 ({t_1} + {z_2} {t_4})^2 \right) X^2 \\&+ 
 \dfrac{1}{{R_1}^2 {R_2}^2} \left(({t_2} + {z_1} {t_4})^2 ({t_1} + {z_2} {t_4})^2 - 
    2 {z_1} {z_2} ({t_2} + {z_1} {t_4}) ({t_1} + {z_2} {t_4}) {t_3}^2 + {t_3}^4 \right) 
\end{align*} which is a quadratic polynomial in $Y := X^2$ with discriminant
\begin{align*}
 D&= \dfrac{1}{{R_1}^2 {R_2}^2} \left({R_1}^2 ({t_2} - {z_1} {t_4})^2 + 2 {z_1} {z_2} {R_1} {R_2} {t_3}^2 + 
      {R_2}^2 ({t_1} + {z_2} {t_4})^2)\right)^2 \\&- 
 \dfrac{4}{{R_1}^2 {R_2}^2} \left(({t_2} + {z_1} {t_4})^2 ({t_1} + {z_2} {t_4})^2 - 
    2 {z_1} {z_2} ({t_2} - {z_1} {t_4}) ({t_1} + {z_2} {t_4}) {t_3}^2 + {t_3}^4 \right).
\end{align*} The rank 0 criterion \cite[Proposition 3.7]{HP} tells us that a rank 0 singularity is non-degenerate of focus-focus type if $D<0$ and non-degenerate of elliptic-elliptic type if $D>0$. We consider four different cases:
\begin{itemize}
\item \underline{\textit{Case $N \times N$ and $S \times S$ for $s_1 \neq \tfrac{1}{2}$:}}

In this case, the discriminant becomes 
\begin{align*}
D=& \dfrac{1}{{R_1}^4 {R_2}^4}(1 - 2 {s_1})^2 ({R_2} + {R_1} {s_2} - 
   {R_2} {s_2})^2  \\& \quad \times \left({R_2}^2 (1 - 2 {s_1})^2 (-1 + {s_2})^2 + {R_1}^2 (1 - 2 {s_1})^2 {s_2}^2 \right. \\& \qquad+ \left. 
   2 {R_1} {R_2} (-16 {s_1}^3 + 8 {s_1}^4 - 20 {s_1} (-1 + {s_2}) {s_2} + 
      4 {s_1}^2 (2 + 5 (-1 + {s_2}) {s_2}) \right. \\& \qquad    \left.+ (-1 + {s_2}) {s_2} (1 + 
         8 (-1 + {s_2}) {s_2}))\right).
\end{align*}

The first three factors are always strictly positive. We divide the fourth factor by ${R_1}^2$ and express it as a function of $s_1, s_2$ and $R := \tfrac{R_2}{R_1}$:
\begin{align*}
\bar{D} =& R^2 (1-2s_1)^2 (1-s_2)^2 + (1-2s_1)^2{s_2}^2 + 2R (-16 {s_1}^3 + 8 {s_1}^4 \\&- 20 {s_1} (-1 + {s_2}) {s_2} + 
      4 {s_1}^2 (2 + 5 (-1 + {s_2}) {s_2}) + (-1 + {s_2}) {s_2} (1 + 8 (-1 + {s_2}) {s_2})).
\end{align*} 

Our goal is to show that this smooth factor is positive in the region defined by $0 \leq s_1,s_2 \leq 1$ and $R>1$. The value of $\bar{D}$ at the vertices $R=1$, $(s_1,s_2) \in \{0,1\}^2$ is 1. Now we look at the eight edges of the region:
\begin{enumerate}[\textbullet]
	\item \underline{$R=1,$ $s_1=0,1$}: There are three critical points: a maximum at $s_2=\tfrac{1}{2}$ with value $1$ and two minima at $s_2=\tfrac{1}{4}(2 \pm \sqrt{2})$ with value $\tfrac{3}{4}$. 
	\item \underline{$R=1,$ $s_2=0,1$}: Same as above, but in $s_1=\tfrac{1}{2}$ and $s_1=\tfrac{1}{4}(2 \pm \sqrt{2})$.
	\item \underline{$s_1=0,1$, $s_2=0,1$}: No critical points. 
\end{enumerate}
Now we consider the five faces of the region:
\begin{enumerate}[\textbullet]
	\item \underline{$R=1$}: There are five critical points: a maximum at $(s_1,s_2)=(\tfrac{1}{2},\tfrac{1}{2})$ with value 4 and four saddle points at $(s_1,s_2) = (\tfrac{1}{10}(5 \pm 2 \sqrt{5}), \tfrac{1}{10}(5 \pm 2 \sqrt{5}))$ with value $\tfrac{4}{5}$. 
	\item \underline{$s_1=0,1$}: There is only one critical point at $(s_2,R) = (\tfrac{1}{12}(7 + \sqrt{13}), \tfrac{1}{9}(5 + 2\sqrt{13}))$, which is a saddle with value $\tfrac{587 + 143\sqrt{13}}{1458} >0$. 
	\item \underline{$s_2=0,1$}: There are no critical points.
\end{enumerate}

Finally, there are no local extrema in the interior of the region, $\bar{D}$ just grows indefinitely as $R \to \infty$. We conclude thus that $\bar{D}$ is positive in all the region. This means that the discriminant $D$ is positive, too, and therefore the singularities are non-degenerate and of elliptic-elliptic type. 
\item \underline{\textit{Case $N \times N$ and $S \times S$ for $s_1 = \tfrac{1}{2}$:}}

In this case the discriminant $D$ vanishes, so we may compute instead the characteristic polynomial of $A_L+A_H = \Om^{-1} (D^2L + D^2H)$, which is
\begin{align*}
\dfrac{(1 + 4 {s_2} - 4 {s_2}^2)^4 + 8 {R_1} {R_2} (1 + 4 {s_2} - 4 {s_2}^2)^2 (-1 + X^2) + 
 16 {R_1}^2 {R_2}^2 (1 + X^2)^2}{16 {R_1}^2 {R_2}^2}.
\end{align*} It is also quadratic in $Y=X^2$ and has discriminant
\begin{align*}
\dfrac{4 (1 + 8 {s_2} + 8 {s_2}^2 - 32 {s_2}^3 + 16 {s_2}^4)}{{R_1} {R_2}}>0.
\end{align*} Therefore, the singularities are non-degenerate and of elliptic-elliptic type.

\item \underline{\textit{Case $N \times S$ and $S \times N$ for $s_1\neq  \tfrac{1}{2}$:}}
In this case the discriminant becomes:
\begin{align*}
D&=\dfrac{1}{{R_1}^4 {R_2}^4} (1 - 2 {s_1})^2 ({R_2} + {R_1} {s_2} - 
   {R_2} {s_2})^2  \\& \quad \times  \left({R_2}^2 (1 - 2 {s_1})^2 (-1 + {s_2})^2 + 
   {R_1}^2 (1 - 2 {s_1})^2 {s_2}^2 \right. \\& \qquad \left.- 
   2 {R_1} {R_2} (8 (-1 + {s_1})^2 {s_1}^2 + {s_2} - 
      12 (-1 + {s_1}) {s_1} {s_2} \right. \\& \qquad \left.+ (7 + 12 (-1 + {s_1}) {s_1}) {s_2}^2 - 16 {s_2}^3 + 
      8 {s_2}^4)\right). 
\end{align*} Since $s_1 \neq \tfrac{1}{2}$, the first three factors are always positive. The fourth factor, which we denote by $E$, determines a closed curve $\symbga$. Outside the curve, the discriminant is positive and therefore the singularities are of elliptic-elliptic type. Inside the curve, the discriminant is negative and therefore the singularities $N \times S$ and $S \times N$ are of focus-focus type.

\item \underline{\textit{Case $N \times S$ and $S \times N$ for $s_1=  \tfrac{1}{2}$:}}

In this case the discriminant $D$ vanishes, so we can compute instead the characteristic polynomial of $A_L+A_H = \Om^{-1} (D^2L + D^2H)$, which is
\begin{align*}
\dfrac{(1 + 4 {s_2} - 4 {s_2}^2)^4 - 8 {R_1} {R_2} (1 + 4 {s_2} - 4 {s_2}^2)^2 (-1 + X^2) + 
 16 {R_1}^2 {R_2}^2 (1 + X^2)^2}{16 {R_1}^2 {R_2}^2},
\end{align*} which is also quadratic in $Y=X^2$ and has discriminant
\begin{align*}
-\dfrac{4 (1 + 8 {s_2} + 8 {s_2}^2 - 32 {s_2}^3 + 16 {s_2}^4)}{{R_1} {R_2}}<0.
\end{align*} Therefore, the singularities are non-degenerate and of focus-focus type.
\end{itemize}
\end{proof}

We now look at the singularities of rank 1. 

\begin{pro}
The system \eqref{sysdef} has only singularities of rank 1 that are non-degenerate and of elliptic-regular type, for any choice $(s_1,s_2) \in [0,1]^2$. 
\label{pro:rank1}
\end{pro}
\begin{proof}
We make use of \cite[Proposition 3.14]{HP}, that provides a criterion for the singularities of rank 1. More specifically, let $l$ be the fixed value of the function $L$, i.e., $l:=R_1 z_1 + R_2 z_2$ and set $\vartheta:= \theta_1-\theta_2$, where $\theta_1,\theta_2$ are the polar angles on $\mbS^2 \times \mbS^2$. We consider the symplectic reduction of the system \eqref{sysdef} on the level $L^{-1}(l)$. Note that from \cite[Lemma 3.10]{HP}, rank 1 singularities always satisfy $z_1,z_2 \neq \pm 1$. Define $$\mcB(z_1) := (1-{z_1}^2) \left( 1-\left( \dfrac{l-R_1 z_1}{R_2} \right)^2 \right) = (1-{z_1}^2)(1-{z_2}^2),$$ which is the content of the square root after the transformation of $H$ in the reduced coordinates $(\vartheta,z_1)$, cf.\ equation \eqref{eq:cyl} in the next section. Then the fixed points of rank 1 are non-degenerate and of elliptic-regular type if
\begin{equation}
\dfrac{2t_4 R_1}{t_3 R_2} \cos(\vartheta) > \dfrac{2 \mcB''(z_1) \mcB(z_1) - (\mcB'(z_1))^2}{4(\mcB(z_1))^{\frac{3}{2}}}.
\label{eq:rhs}
\end{equation} 
In our case, $t_4=0$, so the left hand side vanishes (cf.\ Remark \ref{Tparams}). The right hand side is
$$ - \dfrac{{R_1}^2(1-{z_1}^2)^2 + 2z_1 z_2 R_1 R_2 (1-{z_1}^2)(1-{z_2}^2) + {R_2}^2(1-{z_2}^2)^2}{{R_2}^2 (1-{z_1}^2)^{\frac{3}{2}}(1-{z_2}^2)^{\frac{3}{2}} }. $$ The numerator lies always between $-(\al + \be)^2$ and $-(\al - \be)^2$, where $\al := R_1 (1-{z_1}^2)$ and $\be := R_2 (1-{z_2}^2)$, because $-1 < z_1 z_2 < 1$, so it is always negative and the denominator is always positive. Thus, the right hand side of \eqref{eq:rhs} is always negative and the criterion \cite[Proposition 3.14]{HP}
can be applied. We conclude that all singularities of rank 1 are non-degenerate.
\end{proof}

\begin{theo}
The system \eqref{sysdef} is semitoric for almost any choice of coupling parameters $(s_1,s_2) \in [0,1]^2$. It only fails to be semitoric in the piecewise smooth curve defined by $E=0$, where $E$ is defined in equation \eqref{eq:E}. At this curve, the singularities $N \times S$ and $S \times N$ become degenerate.
\label{th:semit}
\end{theo}
\begin{proof}
Immediate from Propositions \ref{p:integrab}, \ref{pro:nff} and \ref{pro:rank1}. The curve $E=0$ is the border of the coloured region in Figure \ref{fig:4X2nff}.
\end{proof}

\begin{co}
The number of focus-focus points invariant is determined by Proposition \ref{pro:nff} as $n_{FF}=2$ or $n_{FF}=0$ and displayed in Figure \ref{fig:4X2nff}.
\label{co:nff}
\end{co}

Figures \ref{ff:arrayOfPics} and \ref{fig:4X2nff} suggest that system \eqref{sysdef} has several symmetries. This will have an effect on the symplectic invariants.

\begin{pro}
We define the following transformations, acting on the base manifold $(M,\om)$ and the system parameters $(R_1,R_2,s_1,s_2)$:
\begin{small}
\begin{align*}
\Psi_1: & (x_1,y_1,z_1,x_2,y_2,z_2;R_1,R_2,s_1,s_2) \mapsto  (-x_1,-y_1,z_1,-x_2,-y_2,z_2;R_1,R_2,s_1,s_2) \\
\Psi_2: & (x_1,y_1,z_1,x_2,y_2,z_2;R_1,R_2,s_1,s_2) \mapsto  (x_1,-y_1,-z_1,x_2,-y_2,-z_2;R_1,R_2,1-s_1,s_2) \\
\Psi_3: & (x_1,y_1,z_1,x_2,y_2,z_2;R_1,R_2,s_1,s_2) \mapsto  (x_2,y_2,z_2,x_1,y_1,z_1;R_2,R_1,s_1,1-s_2) \\
\Psi_4: & (x_1,y_1,z_1,x_2,y_2,z_2;R_1,R_2,s_1,s_2) \mapsto  (-x_1,-y_1,z_1,x_2,y_2,z_2;R_1,R_2,1-s_1,s_2) \\
\Psi_5: & (x_1,y_1,z_1,x_2,y_2,z_2;R_1,R_2,\tfrac{1}{2},s_2)\; \mapsto  (x_1,y_1,z_1,x_2,y_2,z_2;R_1,R_2,\tfrac{1}{2},1-s_2)
\end{align*}
\end{small}They preserve the symplectic form, $\Psi_i^*\om = \om$, and act on the system \eqref{sysdef} as 
\begin{align*}
& {\Psi_i}^*(L,H) = (L,H)  && \text{ for } i=1,3,5 \\
& {\Psi_i}^*(L,H) = (-L,H) && \text{ for } i=2 \\
& {\Psi_i}^*(L,H) = (L,-H) && \text{ for } i=4. 
\end{align*} 
\label{proptrans4x2} 
\vspace{-1cm} 
\end{pro}
\begin{proof}
Simple substitution in equation \eqref{sysdef}.\\
\end{proof}

\begin{co}
The transformations $\Psi_i$ with $i=1,3,5$ induce semitoric isomorphisms. The systems related by one of these transformations have the same symplectic invariants. 
\label{co:trans}
\end{co}
\begin{proof}
Consequence of the Theorem \ref{th:class}.
\end{proof}

\begin{re}
The transformation $\Psi_2$ induces an isomorphism that reverses the sign of $L$, which resembles the isomorphism that relates the standard and reverse cases of the coupled angular momenta studied in Alonso $\&$ Dullin $\&$ Hohloch \cite{ADH2}. Similarly, the transformation $\Psi_4$ induces an isomorphism that reverses the sign of $H$, which resembles the discrete symmetry of the coupled spin oscillator studied in Alonso $\&$ Dullin $\&$ Hohloch \cite{ADH}.
\end{re}

\subsection{The polygon invariant}

The polygon invariant of the system \eqref{sysdef} can be obtained from the isotropic weights of the function $L$ by making use of the following theorem, that relates the slopes of the polygon to the derivative of the Duistermaat-Heckman function $\rho_L(l)$, that is, the symplectic volume of the reduced space $L^{-1}(l)/\mbS^{1}$.

\begin{theo}[\vungoc \cite{Vu2}]
\label{duist}
Let $\De = f(M)$ be a representative of the polygon invariant. Denote by $\al^+(l)$ the slope of the top boundary of $\De$ and $\al^-(l)$ the bottom boundary. Then the derivative of $\rho_L$ is given by
$$ \rho_L'(l) = \alpha^+(l) - \al^-(l)$$ and it is locally constant on $L(M) \backslash $\emph{Crit}$(L)$, where \emph{Crit}$(L)$ is the set of critical points of $L$. If $\lam \in\! $ \emph{Crit}$(L)$,
$$ \rho_L'(\lam + 0) - \rho_L'(\lam-0) = -c -e^+ - e^-,$$ where $c=1$ if there is a focus-focus point on $L^{-1}(\lam)$ and $c=0$ otherwise. The number $e^+$ ($e^-$ resp.) vanishes if there is no elliptic-elliptic point on the top (bottom resp.) border of $L^{-1}(\lam)$ and otherwise is given by
$$ e^\pm := -\dfrac{1}{a^\pm b^\pm} \geq 0, $$ where $a,b$ are the corresponding isotropic weights.
\end{theo}

Since the phase space $(M,\om)$ and the function $L$ coincide with that of the coupled angular momenta, we can use the isotropic weights computed by Le Floch $\&$ Pelayo \cite{LFPe}. We obtain the following result:

\begin{theo}
\label{th:poly}
The polygon invariant of the system \eqref{sysdef} is determined by the following cases: 
\begin{itemize}
	\item If $\nff=2$, the polygon invariant is the $((\Z_2)^2\times \Z)$-orbit generated by any of the polygons represented in Figure \ref{ff:2FFpolygons}.
	\item If $\nff=0$ and $(s_1,s_2)$ lies in the same connected component as the point $(0,0)$ or $(1,1)$, then the polygon invariant is the $\Z$-orbit generated by the polygon in Figure \ref{ff:2FFpolygons10}.
	\item If $\nff=0$ and $(s_1,s_2)$ lies in the same connected component as the point $(1,0)$ or $(0,1)$, then the polygon invariant is the $\Z$-orbit generated by the polygon in Figure \ref{ff:2FFpolygons01}.
\end{itemize}
\end{theo}

\begin{proof}
We distinguish between two situations:
\begin{itemize}
	\item \textit{\underline{Case $\nff=2$:}} The polygon invariant consists of the quotient of a $((\Z_2)^2\times \Z)$-orbit of a polygon by the $((\Z_2)^2\times \Z)$-action, where the action of $(\Z_2)^2$ comes from the choice of signs  $\epsilon=(\epsilon_1,\epsilon_2)$, one per focus-focus singularity.
	
In this case, the polygons coincide, in principle, with those of the coupled angular momenta, namely Figures \ref{ff:2FFpolygons10} and \ref{ff:2FFpolygons11} but, since we have an additional focus-focus singularity, we also have an additional sign choice, corresponding to the cutting direction on the second singularity. This means that we have to add two more polygons, that is, the ones that coincide with \ref{ff:2FFpolygons10} and \ref{ff:2FFpolygons11} on the left of the image of the second focus-focus singularity and change the slope by one on its right, cf.\ Figures \ref{ff:2FFpolygons00} and \ref{ff:2FFpolygons01}. 
	
	\item \textit{\underline{Case $\nff=0$:}} In this case we do not have any sign choice, so the polygon invariant is the quotient of a $\Z$-orbit by the $\Z$-action. If $\nff=0$ and $(s_1,s_2)$ lies in the same connected component as the point $(0,0)$ or $(1,1)$, by looking at Figure \ref{ff:arrayOfPics} and comparing it with Figure \ref{ff:2FFpolygons}, we identify that the right polygon is Figure \ref{ff:2FFpolygons10}.
	
	If $\nff=0$ and $(s_1,s_2)$ lies in the same connected component as the point $(1,0)$ or $(0,1)$, then the comparison gives the polygon in Figure \ref{ff:2FFpolygons01}.
	\end{itemize}
\end{proof}
\begin{figure}[ht]
\centering
    \begin{subfigure}[b]{0.4\textwidth}
        \includegraphics[width=.9\textwidth]{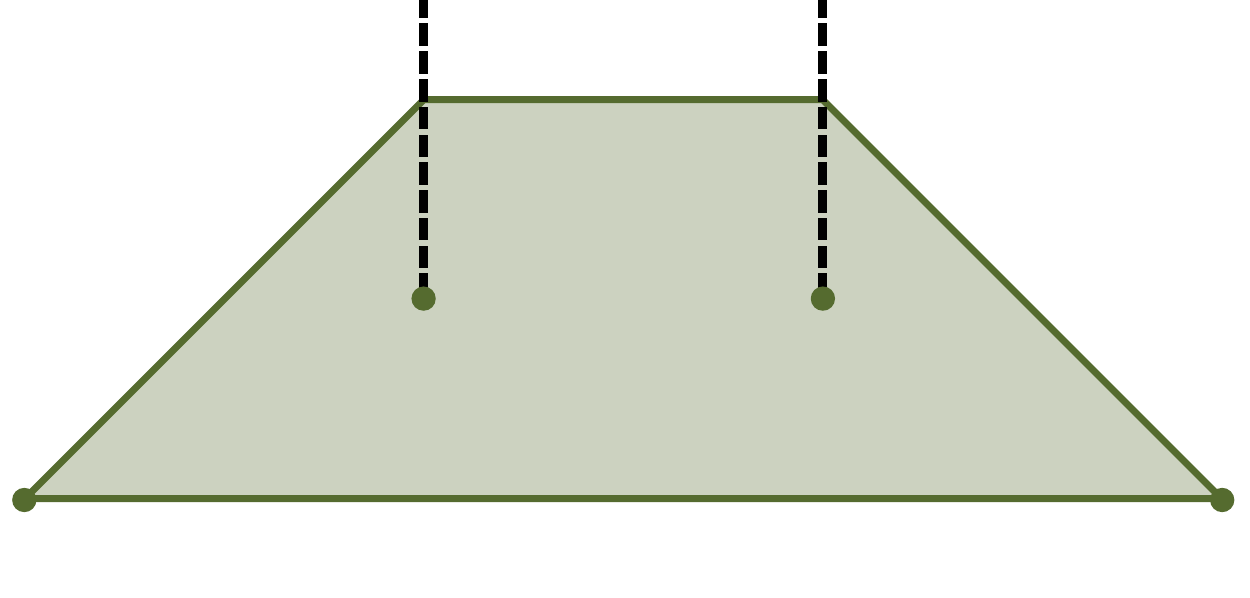}
        \caption{\small \textit{$(\epsilon_1,\epsilon_2)=(+1,+1)$}}
        \label{ff:2FFpolygons00}
    \end{subfigure}
    \hspace{0.1\textwidth}
    \begin{subfigure}[b]{0.4\textwidth}
        \includegraphics[width=.9\textwidth]{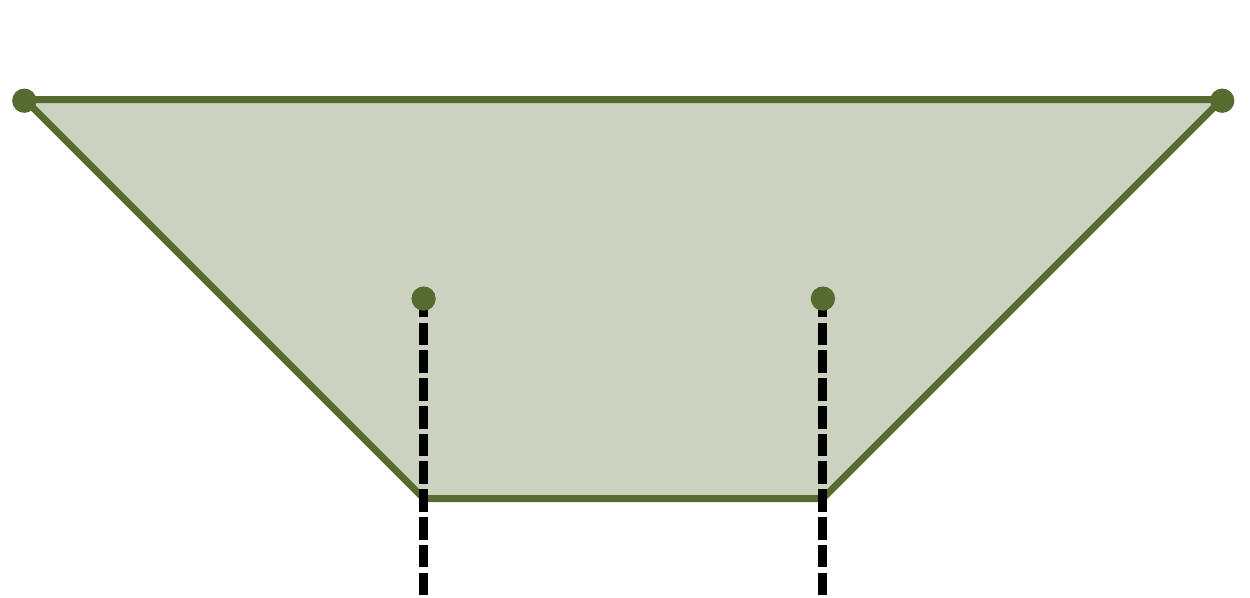}
        \caption{\small \textit{$(\epsilon_1,\epsilon_2)=(+1,-1)$}}
        \label{ff:2FFpolygons10}
    \end{subfigure}    
    \newline
    \begin{subfigure}[b]{0.4\textwidth}
        \includegraphics[width=.9\textwidth]{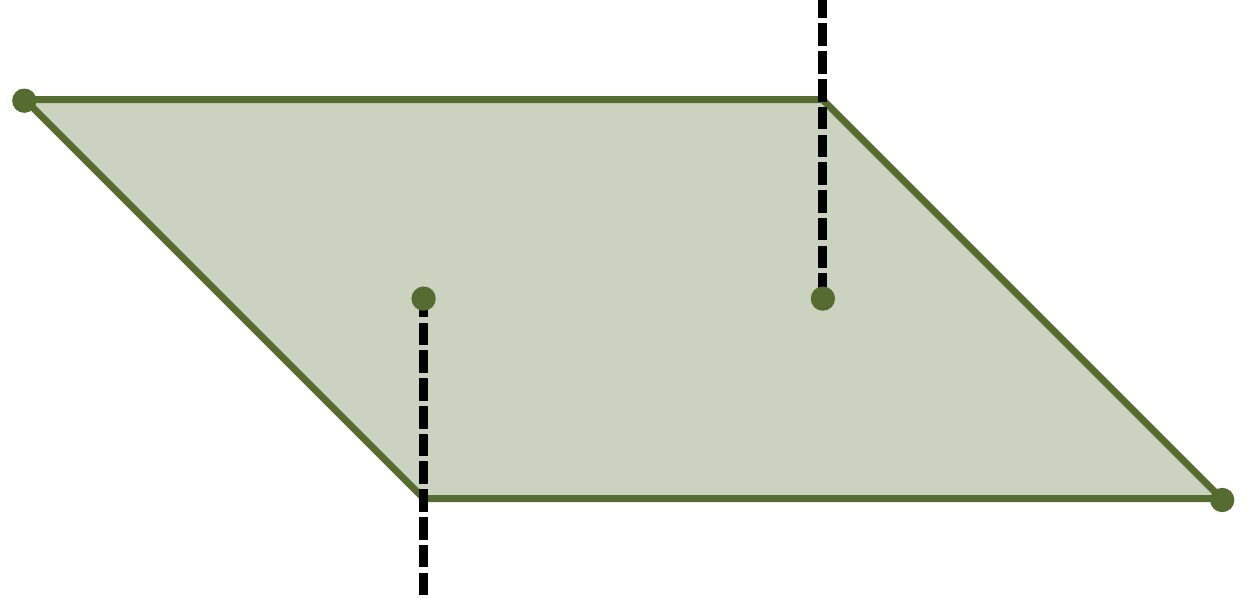}
        \caption{\small \textit{$(\epsilon_1,\epsilon_2)=(-1,+1)$}}
        \label{ff:2FFpolygons01}
    \end{subfigure}
    \hspace{0.1\textwidth}    
    \begin{subfigure}[b]{0.4\textwidth}
        \includegraphics[width=.9\textwidth]{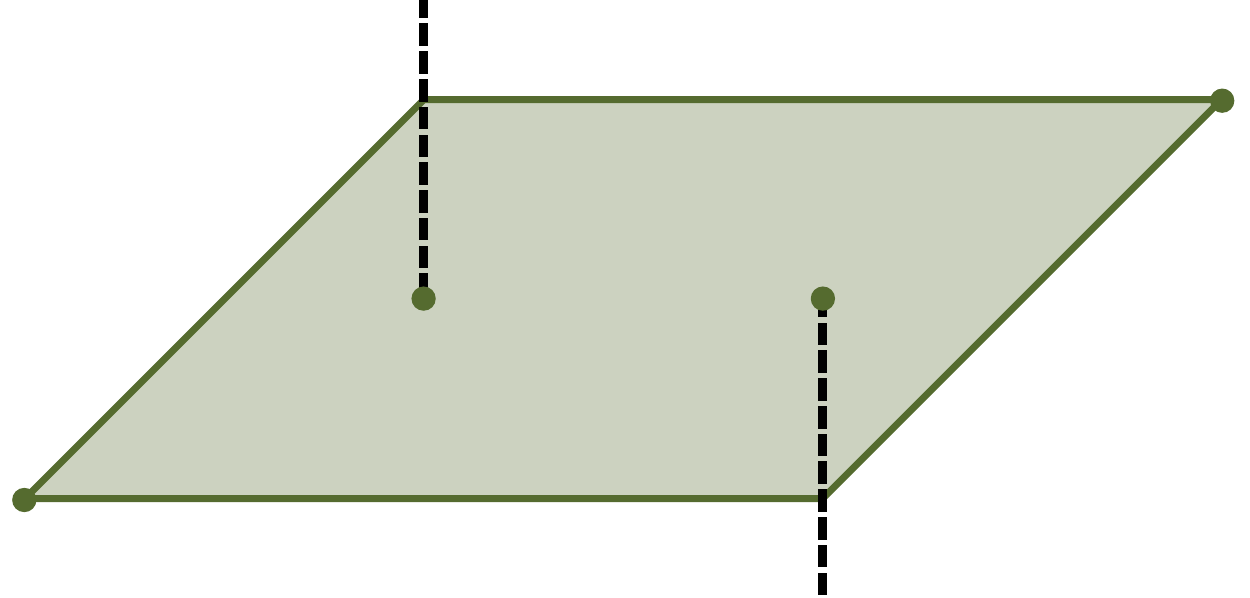}
        \caption{\small \textit{$(\epsilon_1,\epsilon_2)=(-1,-1)$}}
        \label{ff:2FFpolygons11}
    \end{subfigure}
\caption{\textit{\small{Some representatives of the polygon invariant of the system \eqref{sysdef} for different choices of signs $\epsilon=(\epsilon_1,\epsilon_2)$ with the corresponding cutting directions indicated. The horizontal values of the singularities are $L = \pm R_1 \pm R_2$.}}}
\label{ff:2FFpolygons}
\end{figure}

\section{The height invariant}
We compute now the height invariant of the system \eqref{sysdef} for the values of $s_1,s_2$ for which there are two focus-focus singularities. We start by rewriting the system in a more convenient form. Take cylindrical coordinates $(\theta_1,z_1,\theta_2,z_2)$ on $M$, so that the system \eqref{sysdef} becomes
\begin{equation}
\begin{cases}
L(\theta_1,z_1,\theta_2,z_2)\;:=R_1 z_1  + R_2 z_2, \\
H(\theta_1,z_1,\theta_2,z_2):=(1-2s_1)(1-s_2) z_1 + (1-2s_1)s_2 z_2 
\\ \hspace{3.4cm}+ 2(s_1+s_2-{s_1}^2 - {s_2}^2) \sqrt{(1-{z_1}^2)(1-{z_2}^2)}\cos (\theta_1-\theta_2),\\
\end{cases}
\label{eq:cyl}
\end{equation} with the symplectic form $\om = -(R_1\, \dee \theta_1 \wedge \dee z_1  + R_2\, \dee \theta_2 \wedge \dee z_2).$ We now perform the following affine transformation,
\begin{equation*}
\begin{array}{lll}
\qti_1 := -\theta_1 && \qti_2 := \theta_1-\theta_2 \\[0.4cm]
\pti_1 := R_1 z_1 + R_2 z_2 && \pti_2 := R_2 (z_2+1),
\end{array}
\label{change}
\end{equation*} which leads to $L(\qti_1,\pti_1,\qti_2,\pti_2) = \pti_1$ and 
%\todo{\ssc{...again $\times$...}}
\begin{align*}
H(\qti_1,\pti_1,\qti_2,\pti_2)&= \dfrac{1-2s_1}{R_1 R_2} \left( R_2 (\pti_1-\pti_2+R_2)(1-s_2) + s_2 R_1(\pti_2-R_2) \right) \\ & \quad +\dfrac{2(s_1+s_2-{s_1}^2 - {s_2}^2)}{R_1 R_2}  \\ & \qquad \times \sqrt{\pti_2 (\pti_2-2R_2)(\pti_2-\pti_1-R_1-R_2)(\pti_2-\pti_1+R_1-R_2)} \cos(\qti_2),
\end{align*} with symplectic form $\om = \dee \qti_1 \wedge \dee \pti_1 + \dee \qti_2 \wedge \dee \pti_2$. In these coordinates, we see that the function $H$ is independent of $\qti_1$. We simplify now our notation by scaling some of the variables and functions by a factor of $R_1$. We use caligraphic letters to refer to scaled functions and standard letters for unscaled functions:
\begin{equation}
\begin{array}{lllll}
q_1 := \qti_1 && q_2 := \qti_2 && \mcL := \frac{1}{R_1}L\\[0.4cm]
p_1 :=\frac{1}{R_1} \pti_1 && p_2 :=\frac{1}{R_1} \pti_2 && \mcH := H.
\end{array}
\end{equation} This way we can express our problem entirely in terms of $R:= \frac{R_2}{R_1}$. We will do now singular symplectic reduction by the $\mbS^1$-action generated by $L$. We will use two different sets of notations, one adapted to the singularity $N \times S$ and the other to the singularity $S \times N$. 

\subsection{Reduced system for the singularity $N \times S$}
\label{reducedNS}
We start by doing symplectic reduction on the level $\mcL=l+(1-R)$, which is singular in $l=0$, since the focus-focus point $N \times S$ lies on the level set $l=0$. Expressed in coordinates $(q_2,p_2)$, we obtain the reduced Hamiltonian
\begin{equation*}
\begin{aligned}
\mcH^{NS}_l(q_2,p_2) :=& \dfrac{1}{R} (1-2s_1)(R(1+l-2s_2-l s_2)+p_2(s_2-R+R s_2)) \\ 
&\:+ \dfrac{2(s_1+s_2-{s_1}^2 - {s_2}^2)}{R} \sqrt{p_2(p_2-l)(p_2-2R)(p_2-l-2)}\cos( q_2).
\end{aligned}
\end{equation*} 
From this equation we see that the physical region, i.e.\ the domain of definition of $l$ and $p_2$ is given by $l \in [-2,2R]$ and $p_2$ satisfying the inequalities $p_2\geq 0$, $p_2 \geq l$, $p_2 \leq 2R$ and $p_2 \leq l + 2$. The region is depicted in Figure \ref{fig:4X2physNS}. For simplicity, we write 
$$\mcH^{NS}_l(q_2,p_2)  = \mcA^{NS}_l(p_2) + \sqrt{\mcB^{NS}_l(p_2)} \cos(q_2),$$ where
\begin{align*}
\mcA^{NS}_l(p_2) &:= \dfrac{1}{R} (1-2s_1)(R(1+l-2s_2-l s_2)+p_2(s_2-R+R s_2)) \\
\mcB^{NS}_l(p_2) &:= \dfrac{4(s_1+s_2-{s_1}^2 - {s_2}^2)^2}{R^2} p_2(p_2-l)(p_2-2R)(p_2-l-2).
\end{align*}

\begin{figure}[ht]
\centering
\includegraphics[width=7.8cm]{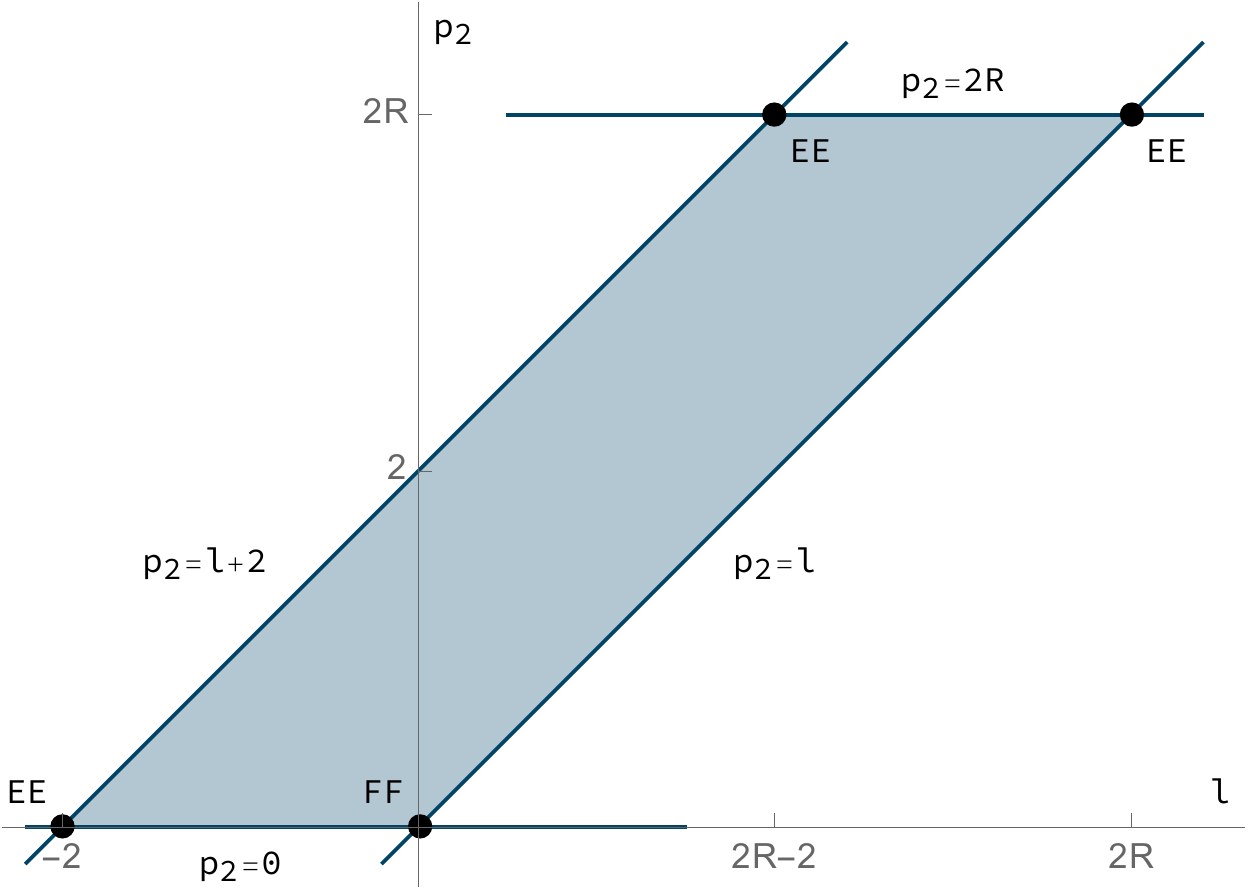}
\caption{\small{\textit{Physical region in the plane $(l,p_2)$ for the reduced model for $N \times S$, corresponding to the values for which $\mcB_l^{NS}(p_2)\geq 0$.}}}
\label{fig:4X2physNS}
\end{figure}

We now define the polynomial
\begin{align*}
	\mcP_l^{NS}(p_2) :=& \mcB_l^{NS}(p_2) -\left(h + (1-2s_1)(1-2s_2)-\mcA_l^{NS} (p_2)\right)^2  \\
	=& \dfrac{1}{R^2} \left(4 (-1 + (1 + l - {p_2})^2) {p_2} ({p_2} - 
      2 R) ((-1 + {s_1}) {s_1} + (-1 + {s_2}) {s_2})^2 \right. \\& \left.- (h R - (-1 + 
        2 {s_1}) (l R (-1 + {s_2}) + {p_2} (R - (1 + R) {s_2})))^2\right)
\end{align*} which is defined in such a way that the singularity $N \times S$ lies precisely on $(l,h)=(0,0)$. The polynomial is of degree 4 in $p_2$.

\subsection{Reduced system for the singularity $S \times N$}
\label{reducedSN}

We focus now on the singularity $S \times N$, so we do symplectic reduction on the level $\mcL=l+(1-R)$, which is singular in $l =0$ since the focus-focus point $S \times N$ lies on the level set $l=0$. We also change the coordinates $(q_2,p_2)$ to
$$ q_2 \mapsto -q_2,\qquad p_2 \mapsto 2R-p_2$$
in order to create a similar notation to the one of \S \ref{reducedNS} 
while preserving the symplectic form. In these new coordinates, the reduced Hamiltonian becomes
\begin{equation*}
\begin{aligned}
\mcH^{SN}_l(q_2,p_2) :=& \dfrac{1}{R} (1-2s_1)(R(-1+l+2s_2-l s_2)+p_2(s_2-R+R s_2)) \\ 
&\:+ \dfrac{2(s_1+s_2-{s_1}^2 - {s_2}^2)}{R} \sqrt{p_2(p_2+l)(p_2-2R)(p_2+l-2)}\cos( q_2).
\end{aligned}
\end{equation*} 

\begin{figure}[ht]
\centering
\includegraphics[width=7.8cm]{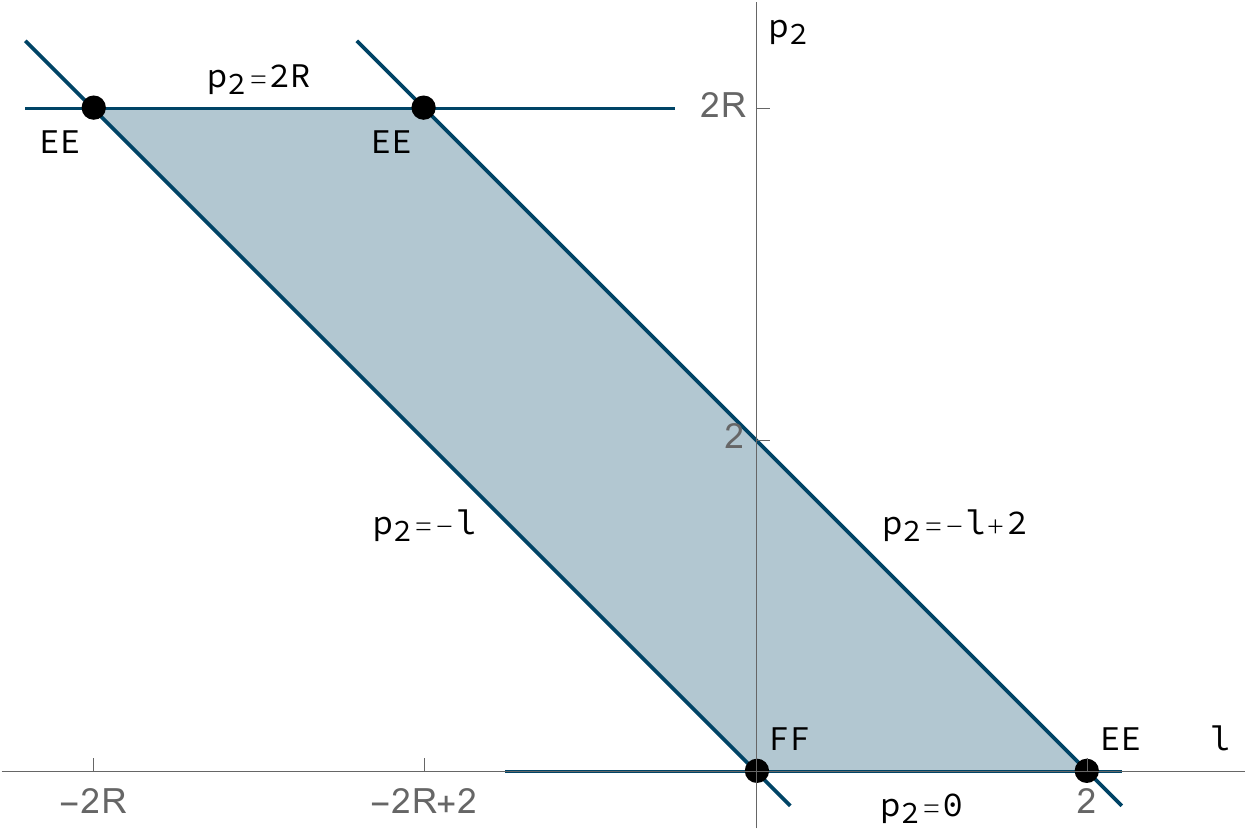}
\caption{\small{\textit{Physical region in the plane $(l,p_2)$ for the reduced model for $S \times N$, corresponding to the values for which $\mcB_l^{SN}(p_2) \geq 0$. Compared to the one corresponding to the reduced model of $N \times S$ in Figure \ref{fig:4X2physNS}, it is reflected on the vertical axis through the focus-focus value.}}}
\label{fig:4X2physSN}
\end{figure}

The physical region in this case will be given by $l \in [-2R,2]$ and $p_2$ satisfying the inequalities $p_2\geq 0$, $p_2 \geq -l$, $p_2 \leq 2R$ and $p_2 \leq -l + 2$. The region is depicted in Figure \ref{fig:4X2physSN}. We write 
$$\mcH^{SN}_l(q_2,p_2)  = \mcA^{SN}_l(p_2) + \sqrt{\mcB^{SN}_l(p_2)} \cos(q_2),$$ where
\begin{align*}
\mcA^{SN}_l(p_2) &:= \dfrac{1}{R} (1-2s_1)(R(-1+l+2s_2-l s_2)+p_2(s_2-R+R s_2)) \\
\mcB^{SN}_l(p_2) &:= \dfrac{4(s_1+s_2-{s_1}^2 - {s_2}^2)^2}{R^2} p_2(p_2+l)(p_2-2R)(p_2+l-2).
\end{align*}

We now define the polynomial
\begin{align*}
	\mcP_l^{SN}(p_2) &:= \mcB_l^{SN}(p_2) -\left(h - (1-2s_1)(1-2s_2)-\mcA_l^{SN} (p_2)\right)^2  \\
	&\,= \dfrac{1}{R^2} \left(4 {p_2} (-2 + l + {p_2}) (l + {p_2}) ({p_2} - 
    2 R) ((-1 + {s_1}) {s_1} + (-1 + {s_2}) {s_2})^2 \right.
    \\ & \;\;\;\;\, \left. - (h R + (-1 + 
      2 {s_1}) ((l + {p_2}) R - ({p_2} + (l + {p_2}) R) {s_2}))^2\right),
\end{align*} in such a way that the singularity $S \times N$ lies precisely on $(l,h)=(0,0)$. The polynomial is also of degree 4 in $p_2$.

\subsection{Computation of the height invariant}
\label{ss:comphei}

Now that we have the two reduced models, the next step is to compute the height invariant $h=(h_1,h_2)$ of the system \eqref{sysdef}. We recall from \S \ref{ss:sinv} that the height $h_r$ associated to the focus-focus singularity $m_r \in M$ is the symplectic volume of 
$$Y_{r}^- := \{ p \in M\; |\; L(p)=L(m_{r}) \text{ and } H(p)< H(m_{r})\},\quad r=1,2.$$

As suggested by Proposition \ref{proptrans4x2}, we are dealing with a very symmetric situation. In particular, the transformation $s_1 \mapsto 1-s_1$ brings the situation of the singularity $N \times S$ to the situation of the singularity $ S \times N$ and vice versa. In Figure \ref{4X2actionsNS}, we can see a plot of this volume for the singularity $N\times S$, so $r=1$ and, in Figure \ref{4X2actionsSN}, we can see the same for the case $r=2$.

\begin{theo}
\label{th:height}
The height invariant $h :=(h_1,h_2)$ associated to the system \eqref{sysdef} for the values of $(s_1,s_2)$ in which it has two focus-focus singularities is given by
\begin{align*}
h_1 &= -\dfrac{1}{2\pi} \mcF(s_1,s_2,R) + 2 u\left( (s_1-\tfrac{1}{2} )(s_2 - \tfrac{R}{R+1}) \right), \\[0.2cm]
h_2&= \hspace{0.3cm} \dfrac{1}{2\pi} \mcF(s_1,s_2,R)  +2 u\left( -(s_1-\tfrac{1}{2} )(s_2 - \tfrac{R}{R+1}) \right) = 2-h_1
\end{align*}
where $R:=\tfrac{R_2}{R_1}$, $u$ is the Heaviside step function and
\begin{align*}
\mcF&(s_1,s_2,R) := 2 R \arctan \left(\frac{\symbga_C}{\sqrt{\symbga_A} (2 {s_1}-1) (R ({s_2}-1)+{s_2})}\right) \\
&+ 2 \arctan\left(\frac{\symbga_D}{\sqrt{\symbga_A} (2 {s_1}-1) (R ({s_2}-1)+{s_2})}\right) \\
       &+ \frac{(2 {s_1}-1) (R ({s_2}-1)+{s_2}) }{2 \left({s_1}^2-{s_1}+{s_2}^2-{s_2}\right)}\log \left(\frac{-\sqrt{\symbga_B}}{2 (R+1) \left({s_1}^2-{s_1}+{s_2}^2-{s_2}\right)+\sqrt{\symbga_A}}\right).
\end{align*} The invariant is represented in Figure \ref{4X2height}. 

\begin{figure}[ht]
 \centering
  	\begin{subfigure}[b]{6.3cm}
        \includegraphics[width=\textwidth]{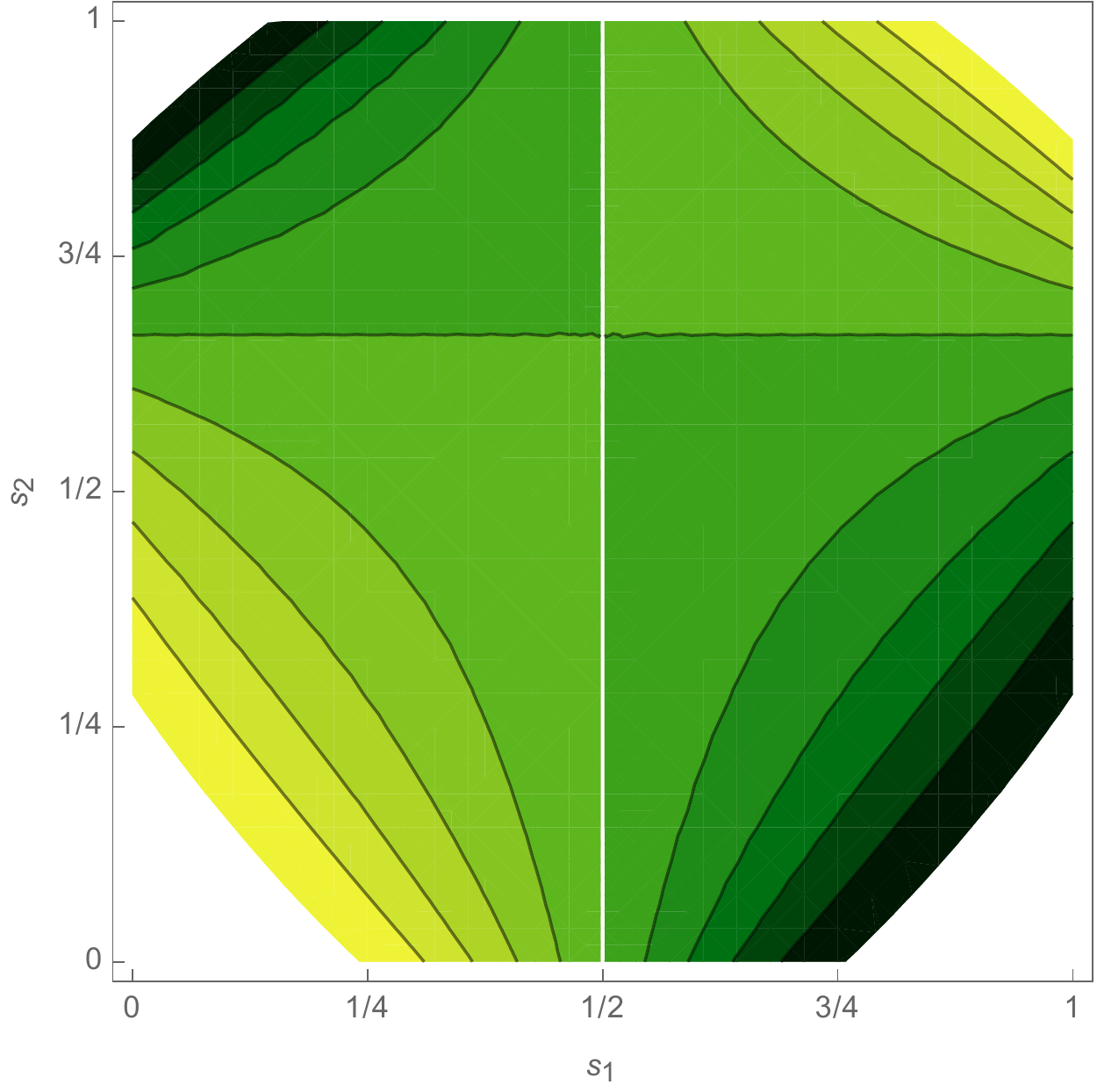}
        \caption{\small \textit{Singularity $N \times S$}}
    \end{subfigure}
    \hspace{1.5cm}
        \begin{subfigure}[b]{6.3cm}
        \includegraphics[width=\textwidth]{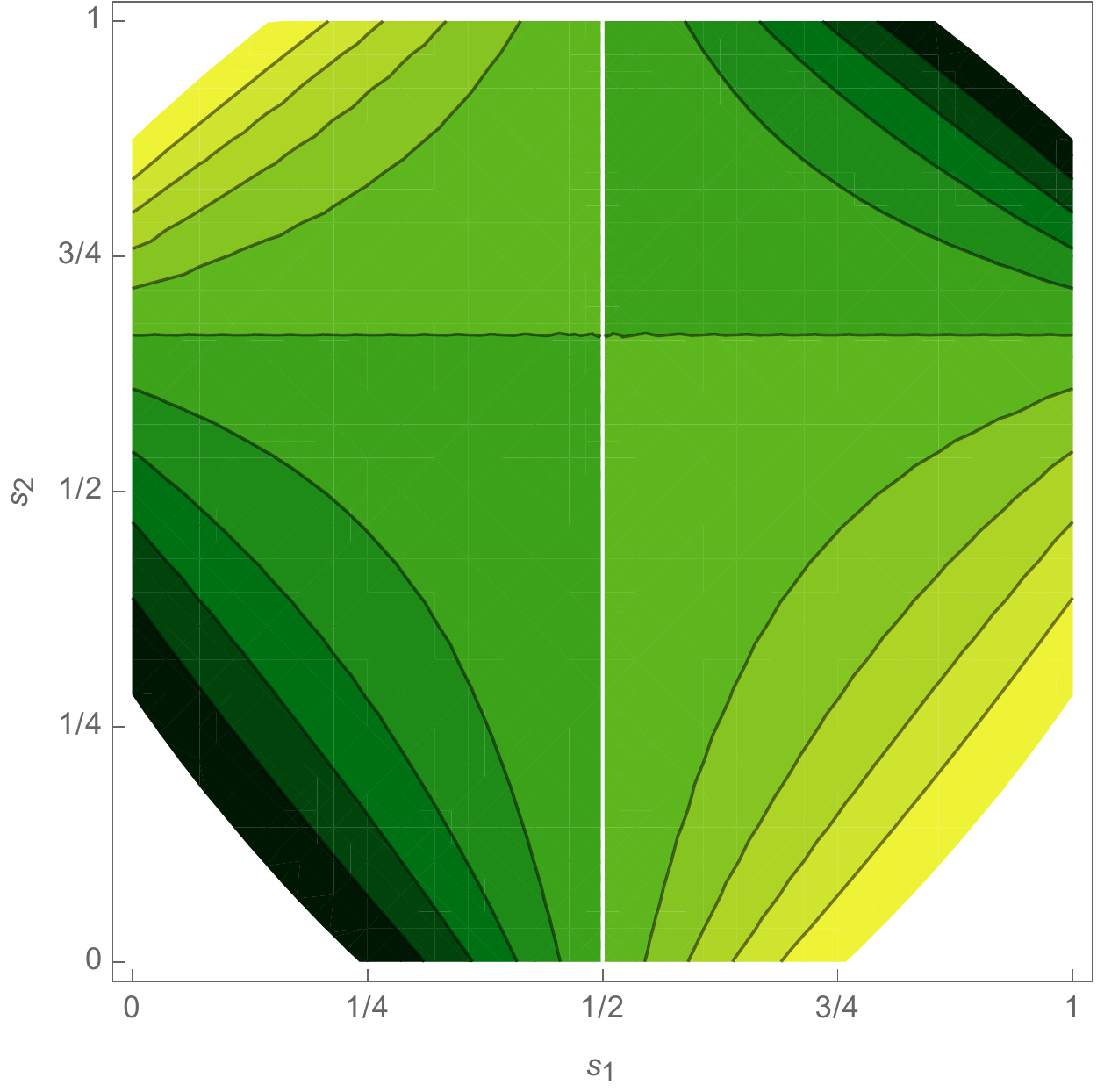}
        \caption{\small \textit{Singularity $S \times N$}}
    \end{subfigure}
\caption{\small \textit{Representation of the height invariant for $R_1=1$, $R_2=2$. Lighter colours represent higher values than darker colours.}}
 \label{4X2height}
\end{figure}

The coefficients $\symbga_A$, $\symbga_B$, $\symbga_C$ and $\symbga_D$ are given by
\begin{align*}
 \symbga_A := & -R^2 (1-2 {s_1})^2 ({s_2}-1)^2+2 R \left(8 {s_1}^4-16 {s_1}^3+4 {s_1}^2 \left(3 {s_2}^2-3 {s_2}+2\right) \right.\\
 &\left.-12 {s_1} ({s_2}-1) {s_2}+{s_2} \left(8 {s_2}^3-16 {s_2}^2+7 {s_2}+1\right)\right)-(1-2 {s_1})^2 {s_2}^2 \\[0.2cm]
 \symbga_B := &  R^2 \left(4 {s_1}^4-8 {s_1}^3+4 {s_1}^2 \left(3 {s_2}^2-4 {s_2}+2\right)-4 {s_1} \left(3 {s_2}^2-4 {s_2}+1\right)\right.\\
 &\left.+({s_2}-1)^2 \left(4 {s_2}^2+1\right)\right)-2 R \left(4 {s_1}^4-8 {s_1}^3+4 {s_1}^2 \left({s_2}^2-{s_2}+1\right) \right.\\
 &\left.-4 {s_1} ({s_2}-1) {s_2}+{s_2} \left(4 {s_2}^3-8 {s_2}^2+3 {s_2}+1\right)\right)+4 {s_1}^4-8 {s_1}^3\\
 &+4 {s_1}^2 \left(3 {s_2}^2-2 {s_2}+1\right)+4 {s_1} {s_2} (2-3 {s_2})+{s_2}^2 \left(4 {s_2}^2-8 {s_2}+5\right) \\[0.2cm]
 \symbga_C := & -4 R^2 {s_1}^2 {s_2}^2+8 R^2 {s_1}^2 {s_2}-4 R^2 {s_1}^2+4 R^2 {s_1} {s_2}^2-8 R^2 {s_1} {s_2}+4 R^2 {s_1}\\
 &-R^2 {s_2}^2+2 R^2 {s_2}-R^2+8 R {s_1}^4-16 R {s_1}^3+8 R {s_1}^2 {s_2}^2-8 R {s_1}^2 {s_2}\\&+8 R {s_1}^2-8 R {s_1} {s_2}^2+8 R {s_1} {s_2}+8 R {s_2}^4-16 R {s_2}^3+6 R {s_2}^2+2 R {s_2}\\
 &+4 \sqrt{{\symbga_B}} \left(-{s_1}^2+{s_1}-{s_2}^2+{s_2}\right)-8 {s_1}^4+16 {s_1}^3-20 {s_1}^2 {s_2}^2+16 {s_1}^2 {s_2}\\
 &-8 {s_1}^2+20 {s_1} {s_2}^2-16 {s_1} {s_2}-8 {s_2}^4+16 {s_2}^3-9 {s_2}^2 \\[0.2cm]
 \symbga_D  :=& -8 R^2 {s_1}^4+16 R^2 {s_1}^3-20 R^2 {s_1}^2 {s_2}^2+24 R^2 {s_1}^2 {s_2}-12 R^2 {s_1}^2+20 R^2 {s_1} {s_2}^2\\&-24 R^2 {s_1} {s_2}+4 R^2 {s_1}-8 R^2 {s_2}^4+16 R^2 {s_2}^3-9 R^2 {s_2}^2+2 R^2 {s_2}-R^2\\
 &+4 R \sqrt{{\symbga_B}} \left(-{s_1}^2+{s_1}-{s_2}^2+{s_2}\right)+8 R {s_1}^4-16 R {s_1}^3+8 R {s_1}^2 {s_2}^2-8 R {s_1}^2 {s_2}\\
 &+8 R {s_1}^2-8 R {s_1} {s_2}^2+8 R {s_1} {s_2}+8 R {s_2}^4-16 R {s_2}^3+6 R {s_2}^2+2 R {s_2}-4 {s_1}^2 {s_2}^2\\
 &+4 {s_1} {s_2}^2-{s_2}^2.
\end{align*}
\label{theo:height}
\end{theo}
\begin{proof}
We start by focussing first on the singularity $N \times S$, i.e.\ $r=1$. We know that $(L,H)(m_1) = (1-R,(1-2s_1)(1-2s_2))$ and we want to compute the symplectic volume of 
$$Y_{1}^- := \{ p \in M\; |\; L(p)=1-R \text{ and } H(p)< (1-2s_1)(1-2s_2)\},$$  which is the area represented in Figure \ref{4X2actionsNS}, divided by $2\pi$. In the notation of \S \ref{reducedNS}, the singularity lies at $(l,h)=(0,0)$. This means that 
$$\mcH^{NS}_0(q_2,p_2)  = \mcA^{NS}_0(p_2) + \sqrt{\mcB^{NS}_0(p_2)} \cos(q_2),$$ where in this case
\begin{align*}
\mcA^{NS}_0(p_2) &:= \dfrac{1}{R} (1-2s_1)(R(1-2s_2)+p_2(s_2-R+R s_2)) \\
\mcB^{NS}_0(p_2) &:= \dfrac{4(s_1+s_2-{s_1}^2 - {s_2}^2)^2}{R^2} {p_2}^2(p_2-2R)(p_2-2)
\end{align*} and $\mcP_0^{NS}(p_2) := \mcB_0^{NS}(p_2) -\left( (1-2s_1)(1-2s_2)-\mcA_0^{NS}(p_2) \right)^2$. The phase space is given by $-\pi \leq q_2\leq \pi$ and $0 \leq p_2 \leq 2$. The roots of $\mcP_0^{NS}(p_2)$ are
\begin{align*}
\ze_1 &= \ze_2=0 \\
\ze_3 & = 1+R - \frac{\sqrt{\symbga_B}}{2({s_1}-{s_1}^2 + s_2 - {s_2}^2)}\\
\ze_4 & = 1+R + \frac{\sqrt{\symbga_B}}{2({s_1}-{s_1}^2 + s_2 - {s_2}^2)},
\end{align*} and the physical region lies between $\ze_2$ and $\ze_3$. There are two trivial cases, $s_1= \tfrac{1}{2}$ and $s_2 = \tfrac{R}{R+1}$. In both cases, $\mcP_0^{NS}(p_2) := \mcB^{NS}_0(p_2)$ and therefore the roots are $\ze_1=\ze_2=0$, $\ze_3=2$ and $\ze_4=2R$. These trivial cases form the border between two different behaviours, i.e., we distinguish the following situations:
\begin{multicols}{2}
\begin{itemize}
	\item Case I: $s_1<\tfrac{1}{2}$ and $s_2 < \tfrac{R}{R+1}$
	\item Case II: $s_1<\tfrac{1}{2}$ and $s_2 > \tfrac{R}{R+1}$
	\item Case III: $s_1=\tfrac{1}{2}$ or $s_2 = \tfrac{R}{R+1}$
	\item Case IV: $s_1>\tfrac{1}{2}$ and $s_2 < \tfrac{R}{R+1}$
	\item Case V: $s_1>\tfrac{1}{2}$ and $s_2 > \tfrac{R}{R+1}$
\end{itemize}
\end{multicols}

\begin{figure}[ht]
 \centering
  	\begin{subfigure}[b]{4cm}
        \includegraphics[width=4cm]{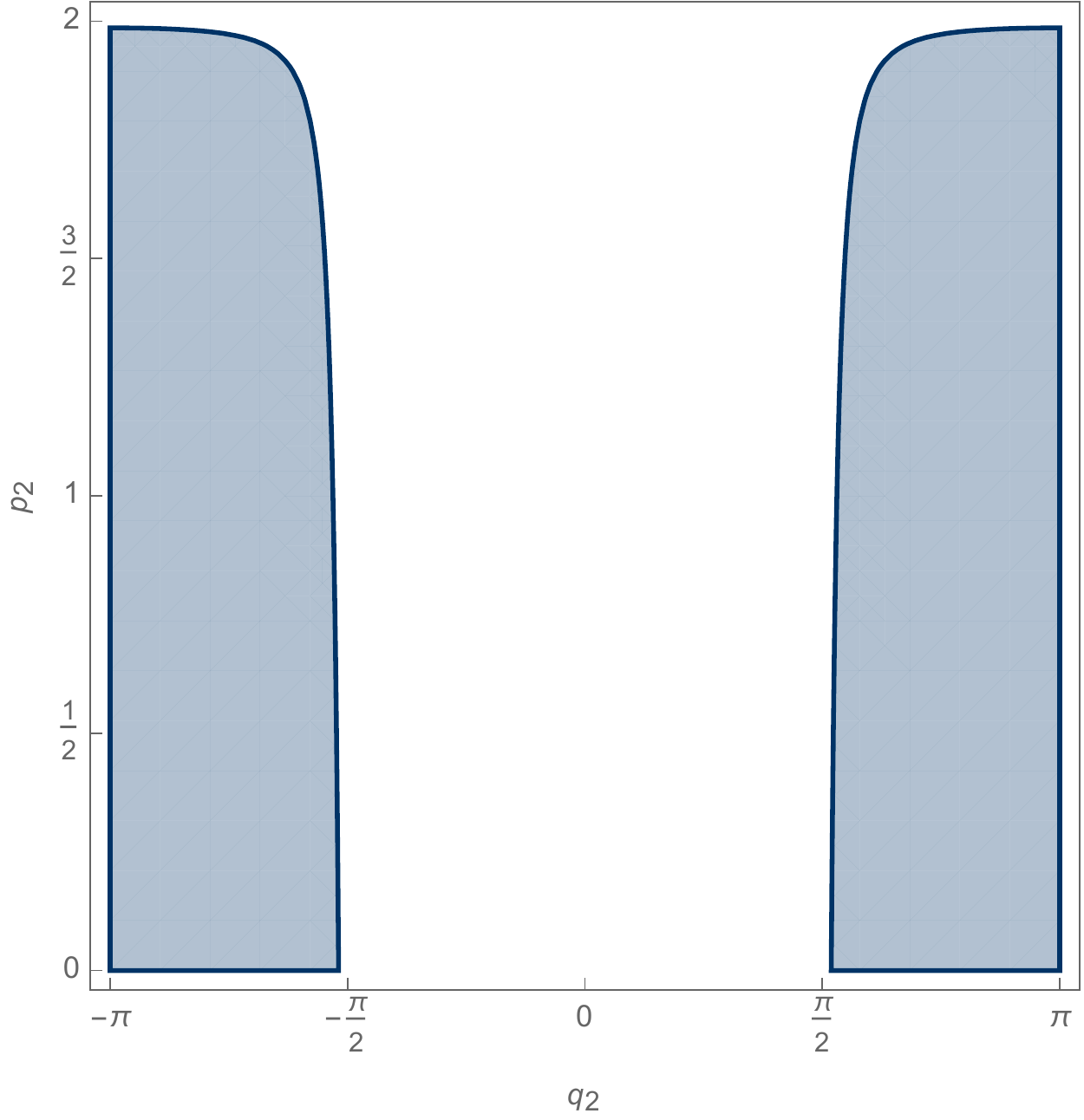}
        \caption*{$s_1 = \tfrac{1}{4}$, $s_2=\tfrac{3}{4}$}
    \end{subfigure}
    \hfill
        \begin{subfigure}[b]{4cm}
        \includegraphics[width=4cm]{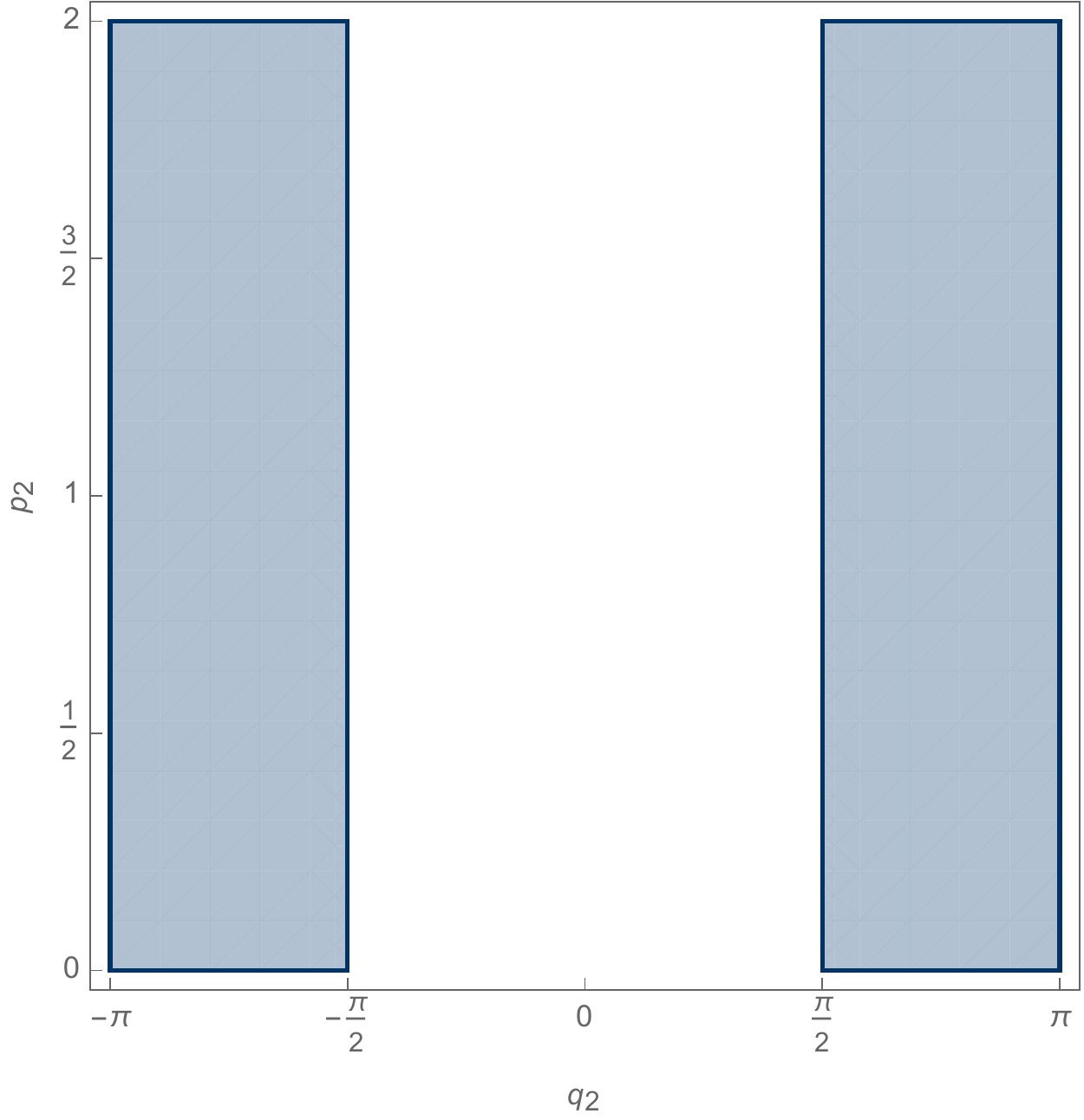}
        \caption*{$s_1 = \tfrac{1}{2}$, $s_2=\tfrac{3}{4}$}
    \end{subfigure}
    \hfill
    \begin{subfigure}[b]{4cm}
        \includegraphics[width=4cm]{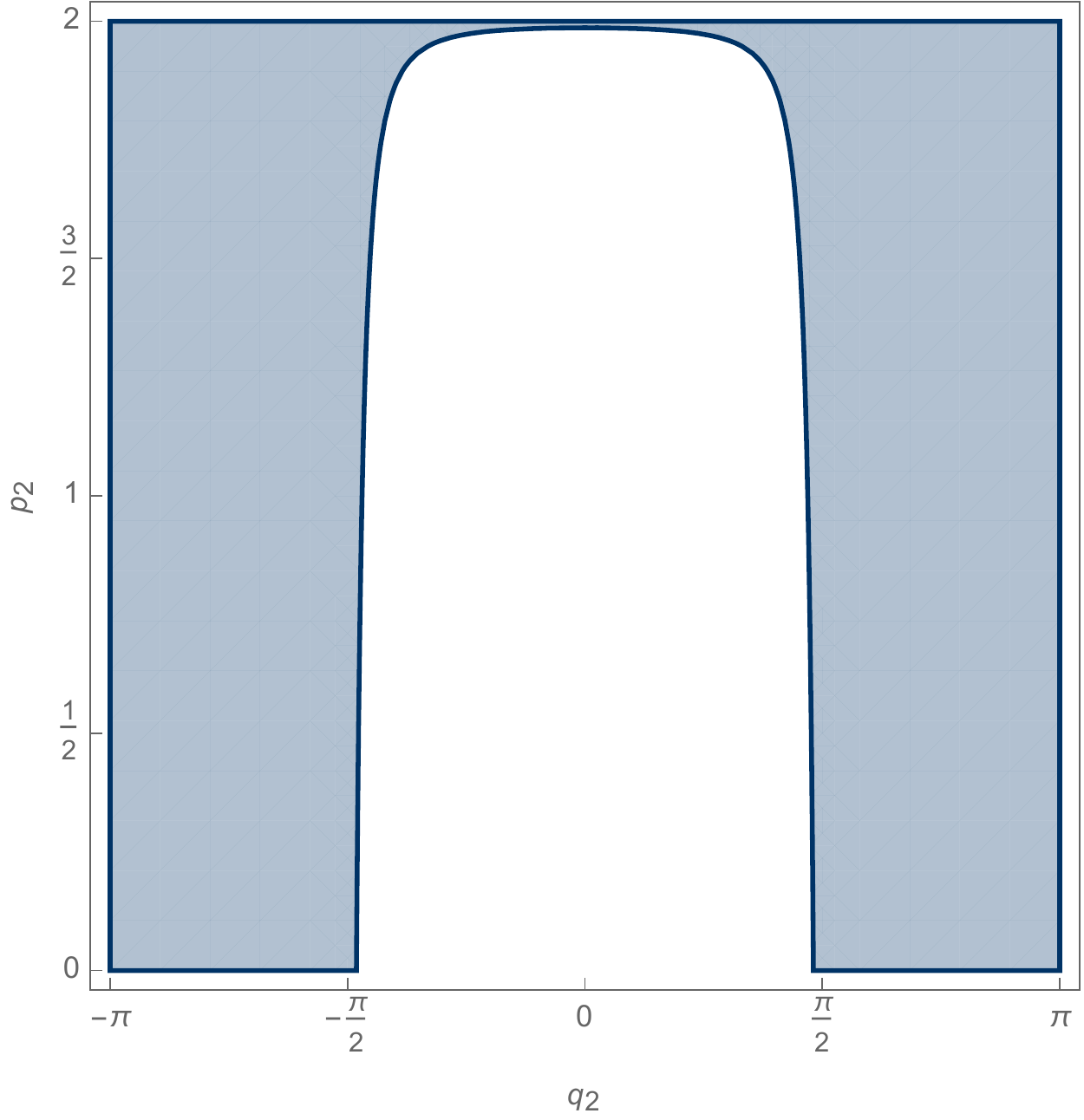}
        \caption*{$s_1 = \tfrac{3}{4}$, $s_2=\tfrac{3}{4}$}
    \end{subfigure}
    \\[0.4cm]
    \begin{subfigure}[b]{4cm}
        \includegraphics[width=4cm]{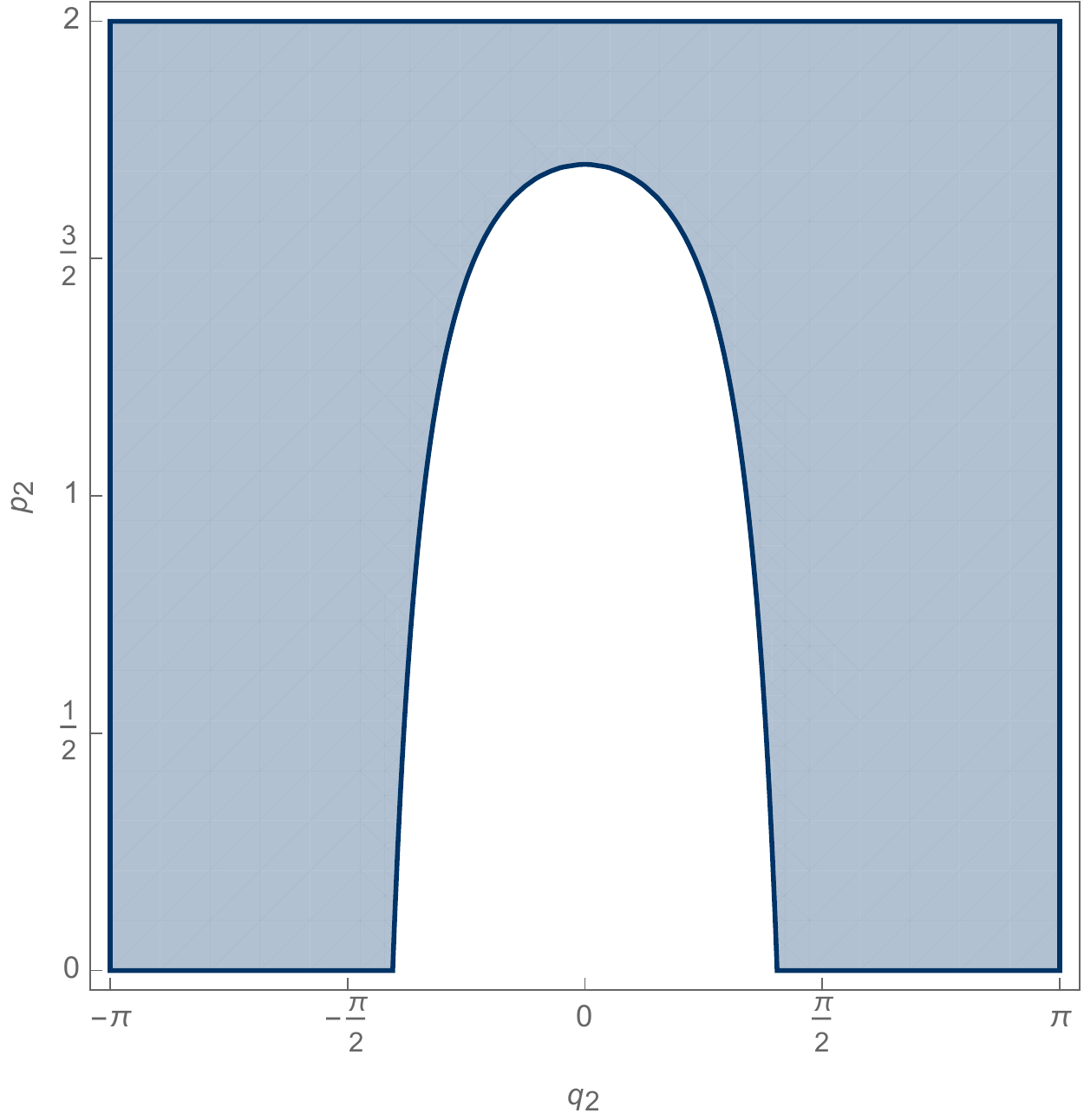}
        \caption*{$s_1 = \tfrac{1}{4}$, $s_2=\tfrac{1}{4}$}
    \end{subfigure}
    \hfill
        \begin{subfigure}[b]{4cm}
        \includegraphics[width=4cm]{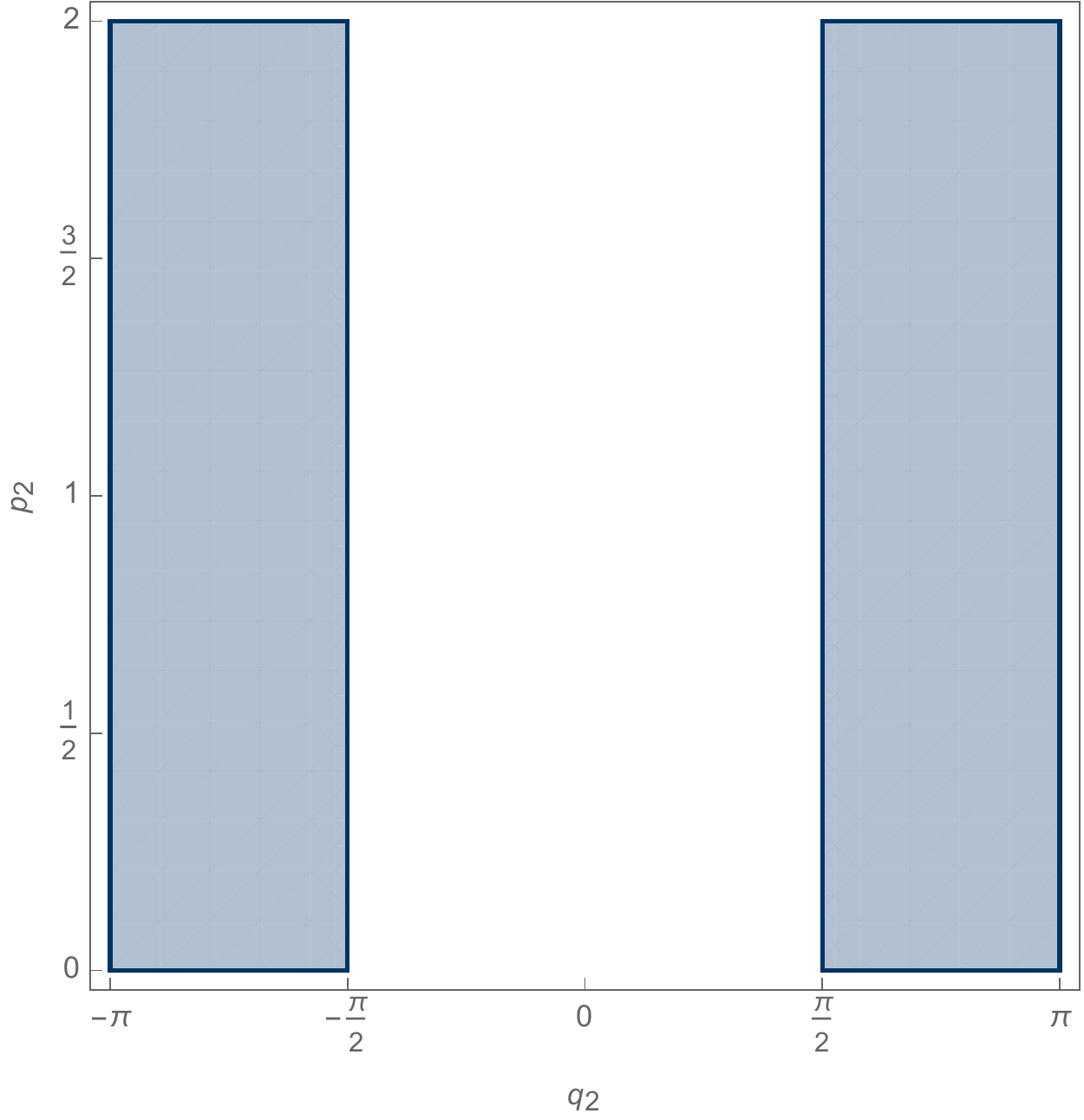}
        \caption*{$s_1 = \tfrac{1}{2}$, $s_2=\tfrac{1}{4}$}
    \end{subfigure}
    \hfill
    \begin{subfigure}[b]{4cm}
        \includegraphics[width=4cm]{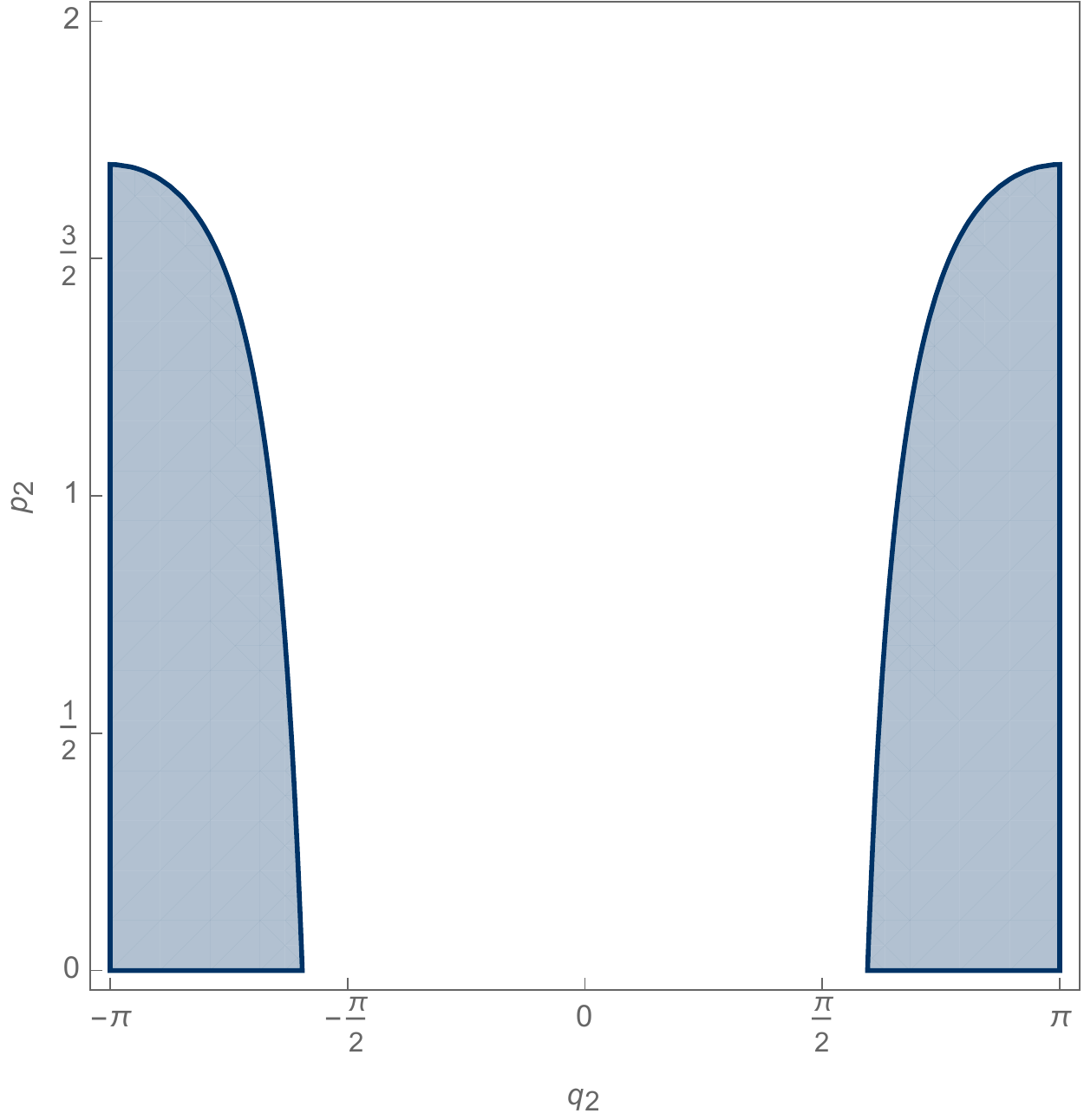}
        \caption*{$s_1 = \tfrac{3}{4}$, $s_2=\tfrac{1}{4}$}
    \end{subfigure}
\caption{\small \textit{Area corresponding to the height invariant of the singularity $N \times S$ for $R=2$.}}
 \label{4X2actionsNS}
\end{figure}
If we look at Figure \ref{4X2actionsNS}, we see that cases I and V, bottom-left and top-right respectively, are connected from above. Case III, in the centre, corresponds to the trivial transition situation. Cases II and IV, top-left and bottom-right respectively, are not connected from above. Therefore, for cases I and V we write
\begin{align*}
h_1 &= \dfrac{1}{2\pi} \left( 4\pi -  \oint_{h=0} \arccos \dfrac{ (1-2s_1)(1-2s_2)-\mcA^{NS}_0(p_2)}{\sqrt{\mcB^{NS}_0(p_2)}} \dee p_2   \right) \\&= \dfrac{1}{2\pi} \left( 4\pi-2 \int_{\ze_2}^{\ze_3}  \arccos \dfrac{ (1-2s_1)(1-2s_2)-\mcA^{NS}_0(p_2)}{\sqrt{\mcB^{NS}_0(p_2)}} \dee p_2 \right)
\end{align*} and for the cases II and IV we write 
\begin{align*}
h_1 &= \dfrac{1}{2\pi} \left( 4\pi -  \oint_{h=0} \arccos \dfrac{ (1-2s_1)(1-2s_2)-\mcA^{NS}_0(p_2)}{\sqrt{\mcB^{NS}_0(p_2)}} \dee p_2 -2\pi (2-\ze_3)  \right) \\&= \dfrac{1}{2\pi} \left( 4\pi-2 \int_{\ze_2}^{\ze_3}  \arccos \dfrac{ (1-2s_1)(1-2s_2)-\mcA^{NS}_0(p_2)}{\sqrt{\mcB^{NS}_0(p_2)}} \dee p_2 -2\pi (2-\ze_3)  \right).
\end{align*} 
\begin{figure}[ht]
 \centering
  	\begin{subfigure}[b]{4cm}
        \includegraphics[width=4cm]{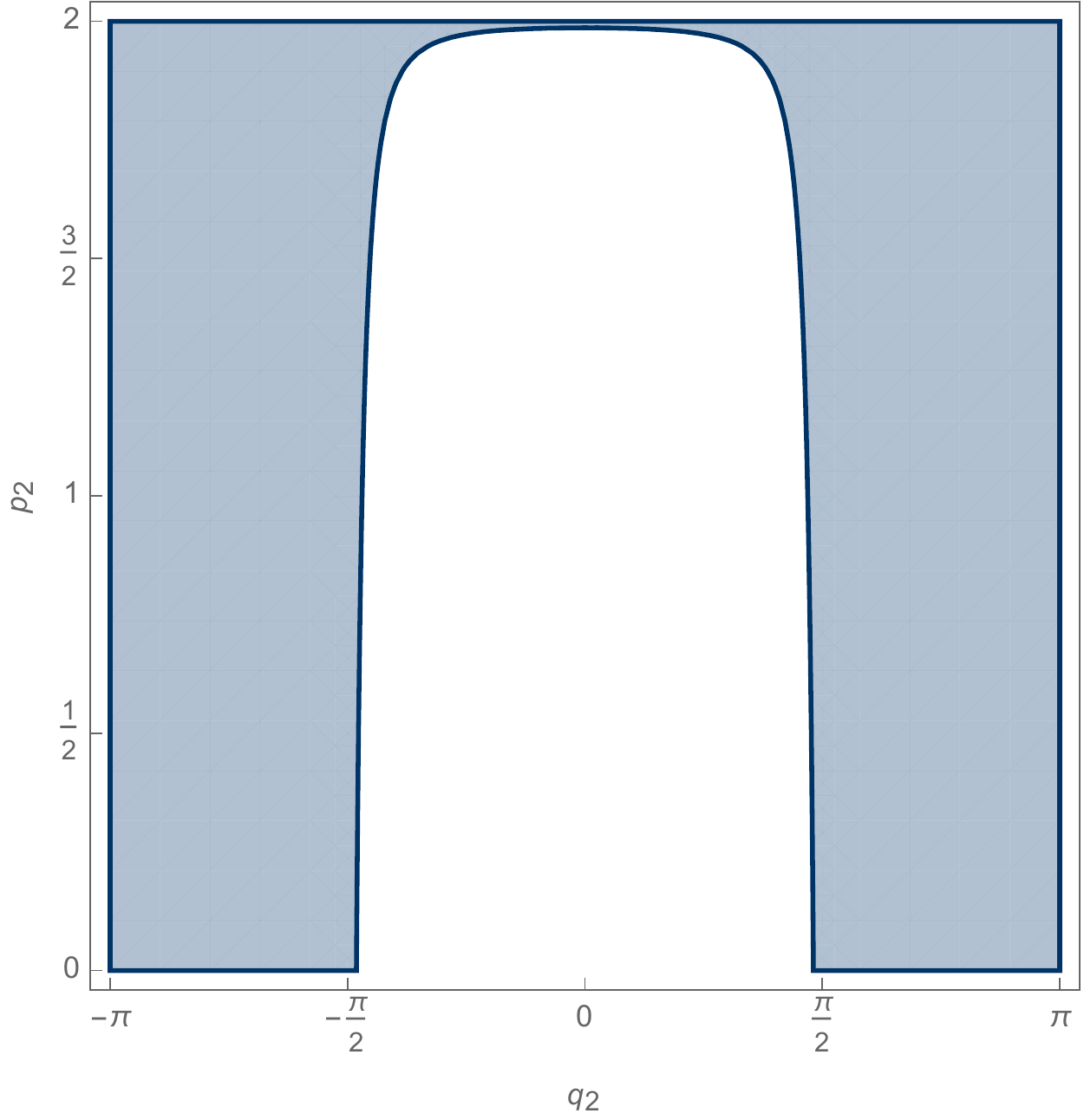}
        \caption*{$s_1 = \tfrac{1}{4}$, $s_2=\tfrac{3}{4}$}
    \end{subfigure}
    \hfill
        \begin{subfigure}[b]{4cm}
        \includegraphics[width=4cm]{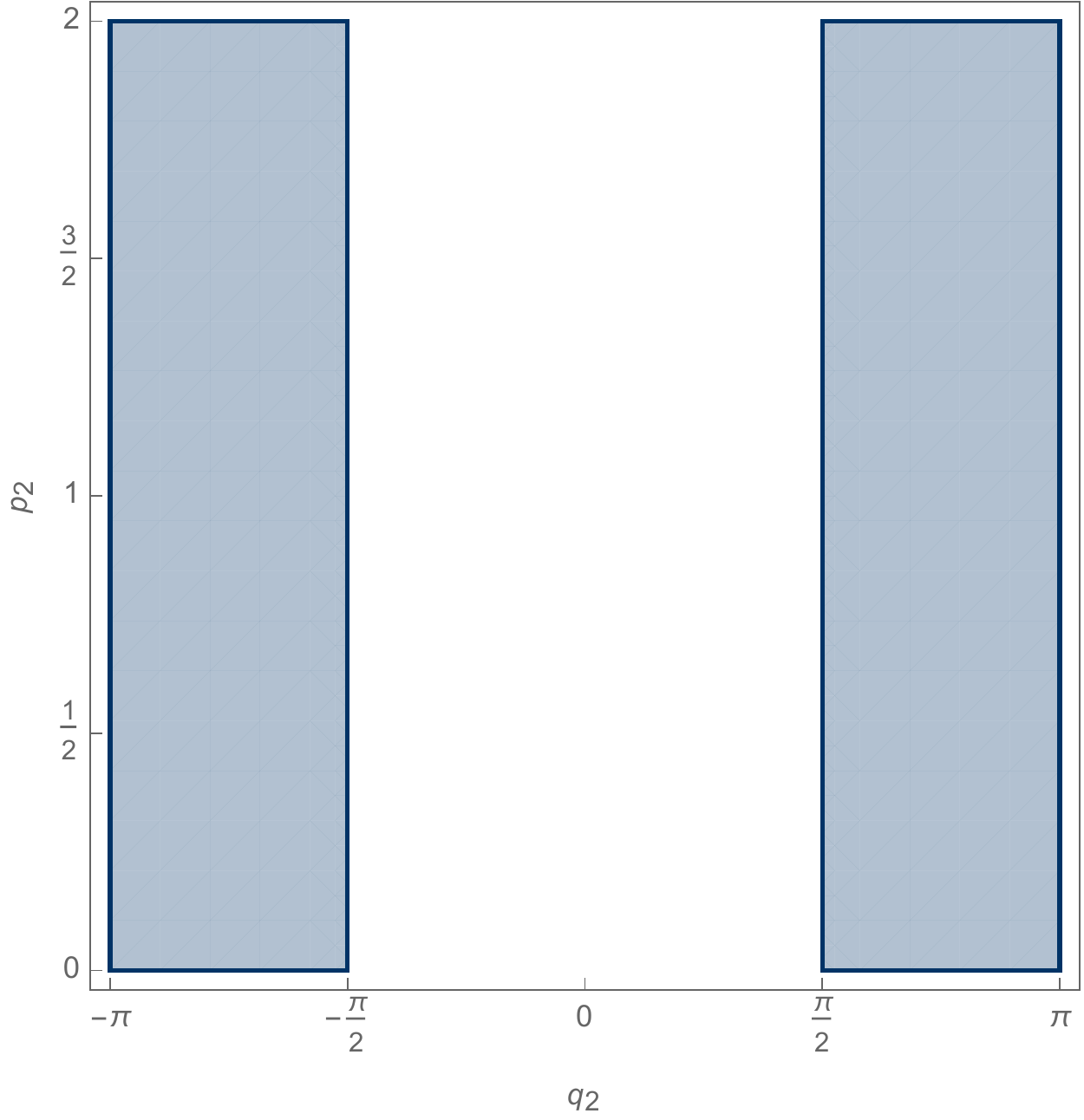}
        \caption*{$s_1 = \tfrac{1}{2}$, $s_2=\tfrac{3}{4}$}
    \end{subfigure}
    \hfill
    \begin{subfigure}[b]{4cm}
        \includegraphics[width=4cm]{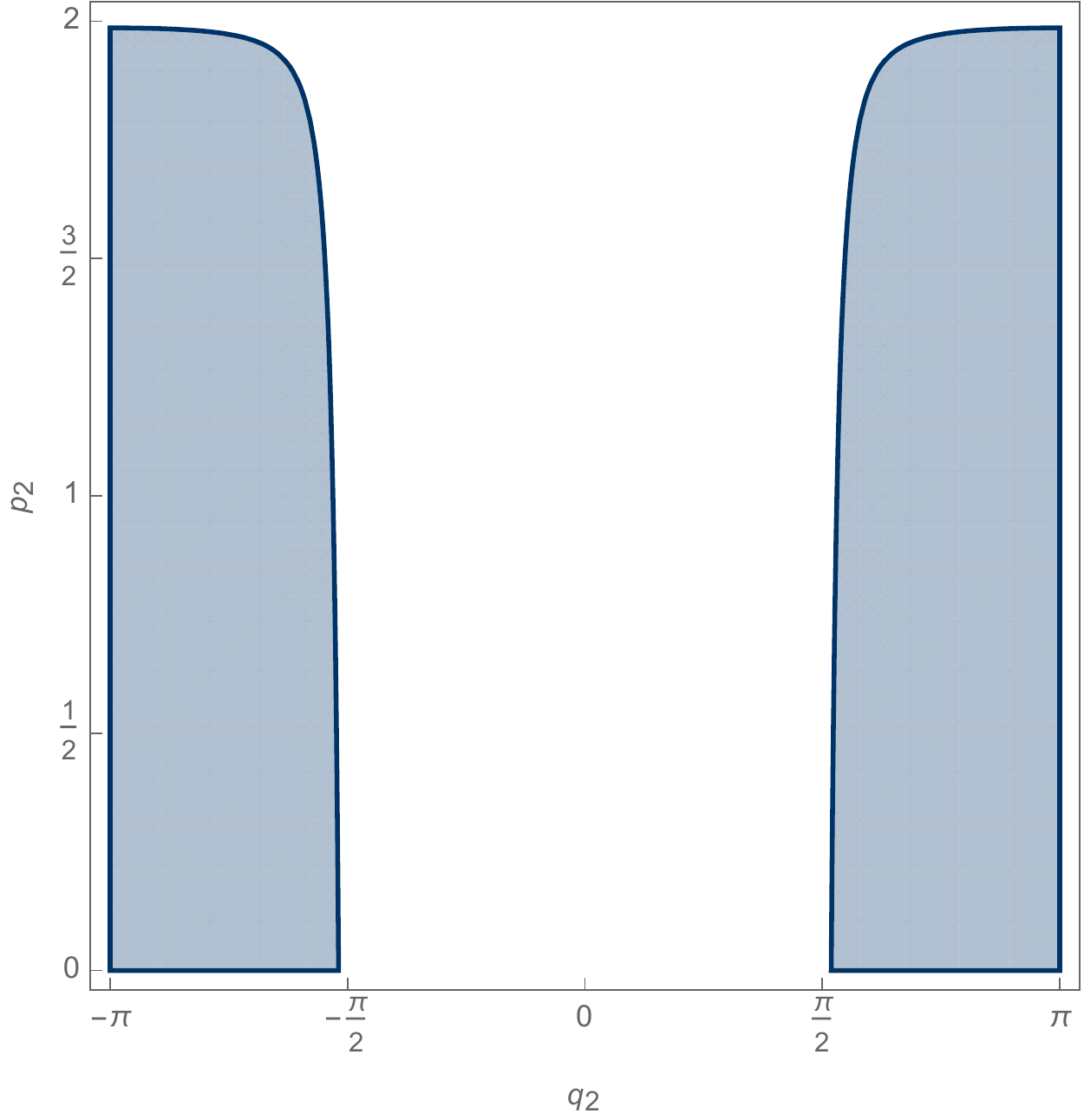}
        \caption*{$s_1 = \tfrac{3}{4}$, $s_2=\tfrac{3}{4}$}
    \end{subfigure}
    \\[0.4cm]
    \begin{subfigure}[b]{4cm}
        \includegraphics[width=4cm]{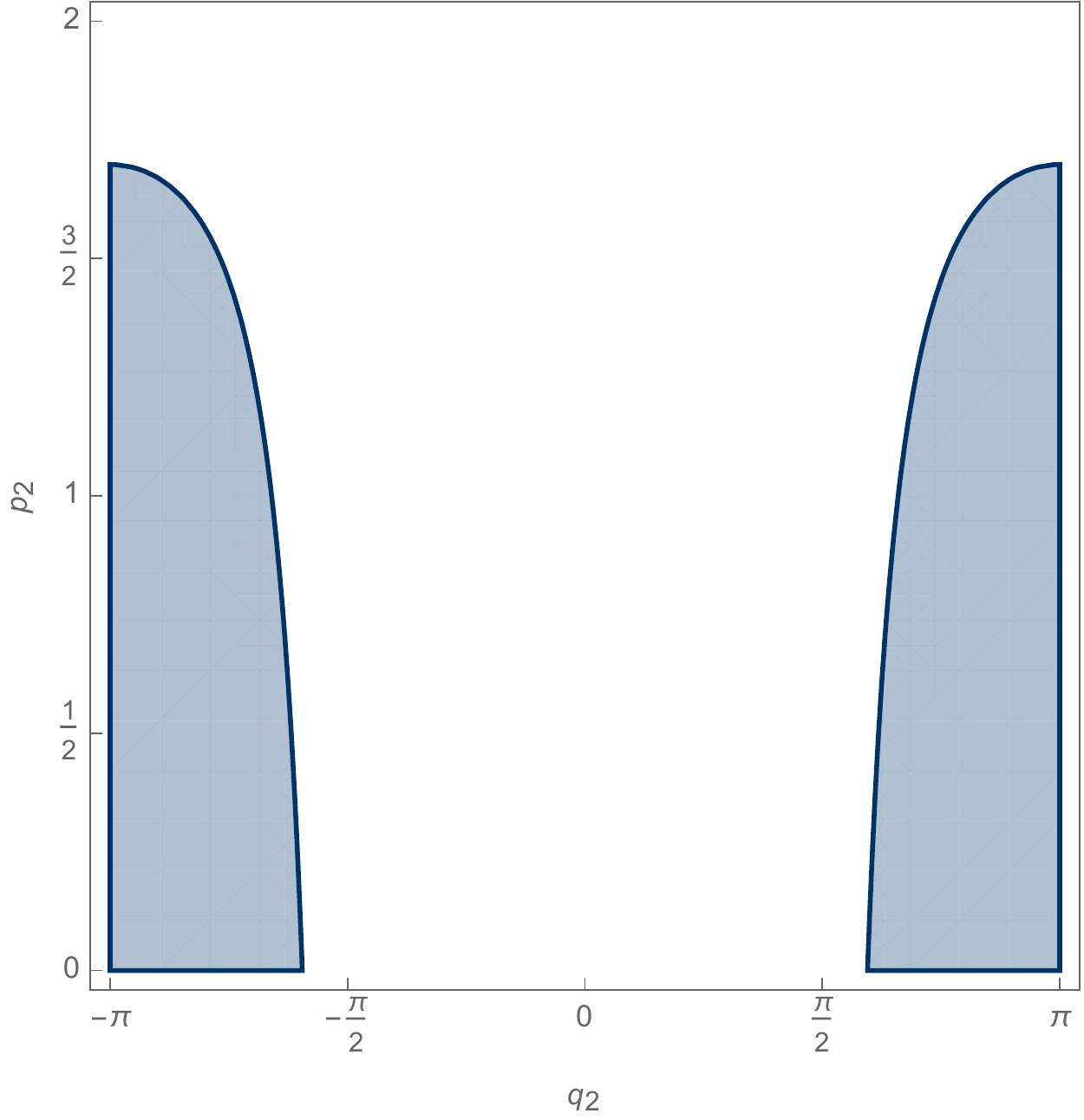}
        \caption*{$s_1 = \tfrac{1}{4}$, $s_2=\tfrac{1}{4}$}
    \end{subfigure}
    \hfill
        \begin{subfigure}[b]{4cm}
        \includegraphics[width=4cm]{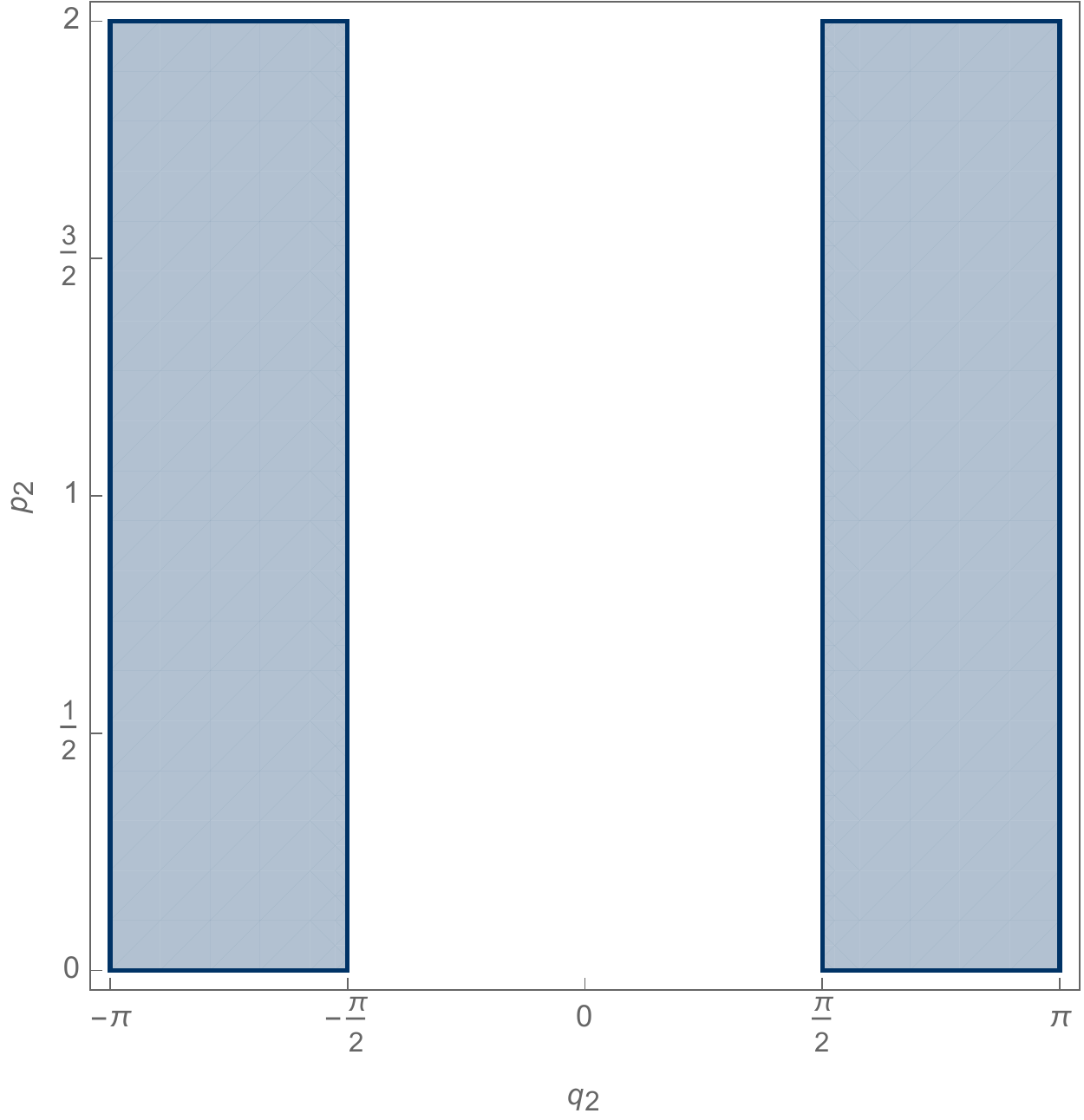}
        \caption*{$s_1 = \tfrac{1}{2}$, $s_2=\tfrac{1}{4}$}
    \end{subfigure}
    \hfill
    \begin{subfigure}[b]{4cm}
        \includegraphics[width=4cm]{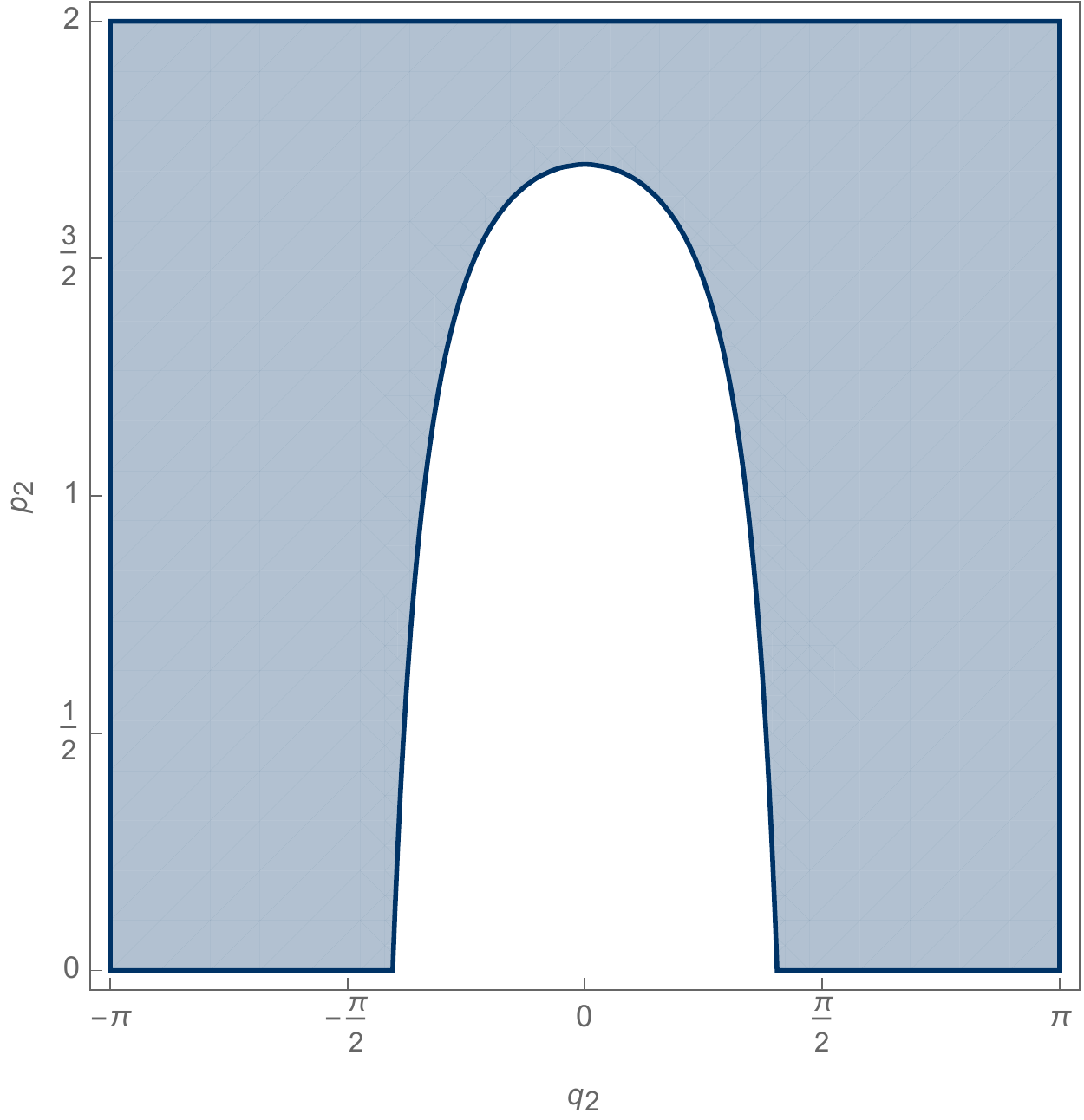}
        \caption*{$s_1 = \tfrac{3}{4}$, $s_2=\tfrac{1}{4}$}
    \end{subfigure}
\caption{\small \textit{Area corresponding to the height invariant of the singularity $S \times N$ for $R=2$.}}
 \label{4X2actionsSN}
\end{figure}

For the trivial case III we have
$$ h_1 = \dfrac{1}{2\pi} \left( \dfrac{1}{2} 4\pi \right) = 1.$$
As we did in for the general action integral in \S \ref{reducedNS}, we integrate by parts to compute this integral. For the cases I and V we obtain
$$ h_1 = \dfrac{1}{2\pi} \left( 4\pi - \mcF(s_1,s_2,R) \right)$$ and for the cases II and IV,
$$ h_1 = \dfrac{1}{2\pi} \left( -\mcF(s_1,s_2,R) \right),$$ where
\begin{equation}
\mcF(s_1,s_2,R) := \dfrac{2}{2\pi} \int_{\ze_2}^{\ze_3} \dfrac{V(p_2)}{\sqrt{Q(p_2)}} \dee p_2.
\label{int4X2}
\end{equation} The rational function $V(p_2)$ and the polynomial $Q(p_2)$ are given by
\begin{align*}
V(p_2) &=-(2 {s_1}-1) (R {s_2}-R+{s_2}) +\frac{-2 R {s_1} {s_2}+2 R {s_1}+R {s_2}-R-2 {s_1} {s_2}+{s_2}}{{p_2}-2} \\&+ \frac{-2 R^2 {s_1} {s_2}+2 R^2 {s_1}+R^2 {s_2}-R^2-2 R {s_1} {s_2}+R {s_2}}{{p_2}-2 R}\\
Q(p_2) &=4 ({p_2}-2) ({p_2}-2 R) (({s_1}-1) {s_1}+({s_2}-1) {s_2})^2\\&-(1-2 {s_1})^2 (R ({s_2}-1)+{s_2})^2.
\end{align*} Here it is important to observe that the polynomial $Q(p_2)$ is of degree 2 in $p_2$, so the integral \eqref{int4X2} is not elliptic but can be solved explicitly in terms of elementary functions. We will need the following definite integrals
\begin{align*}
\mcN_A(\al,\be,\symbga):=\int_0^{\tfrac{-\be-\sqrt{\be^2-4 \al \symbga}}{2\al}} \dfrac{\dee x}{\sqrt{\al x^2 + \be x + \symbga}} = \dfrac{1}{\sqrt{\al}} \log \left( \dfrac{- \sqrt{\be^2 - 4 \al \symbga}}{\be + 2\sqrt{\al \symbga}} \right) 
\end{align*} and 
\begin{align*}
\mcN_B(\al,\be,\symbga,\delta ):=&\int_0^{\tfrac{-\be-\sqrt{\be^2-4 \al \symbga}}{2\al}} \dfrac{\dee x}{(\delta-x)\sqrt{\al x^2 + \be x + \symbga}} \\=& \dfrac{1}{\sqrt{\symbga + \delta(\be + \al \delta)}} \log \left( \dfrac{-2\symbga - \be \delta + 2 \sqrt{\symbga^2 + \symbga \delta (\beta + \al \delta)}}{\delta \sqrt{\be^2-4 \al \symbga }} \right) \\
=& \dfrac{2}{\sqrt{-\symbga - \delta (\be + \al \delta)}} \arctan \left( \dfrac{2 \symbga + \delta \left( \be + \sqrt{\be^2-4\al \symbga} \right)}{2 \sqrt{-\symbga(\symbga + \delta(\be + \al \delta))}} \right).  
\end{align*} In the last step we have used the identity $\arctan(z) = \tfrac{i}{2} \left( \log(1-iz) - \log(1+iz) \right)$ with 
$$ z = -\dfrac{i\left( 2 \symbga + \delta \left( \be + \sqrt{\be^2-4\al \symbga} \right)\right)}{2 \sqrt{-\symbga(\symbga + \delta(\be + \al \delta))}}.$$
We rewrite now the integral \eqref{int4X2} as
\begin{align*}
\mcF(s_1,s_2,R) &=  \dfrac{1}{\pi} \int_0^{\tfrac{-\be-\sqrt{\be^2-4 \al \symbga}}{2\al}} \left( V_1 + \dfrac{V_2}{2-p_2} + \dfrac{V_3}{2R-p_2}\right) \dfrac{\dee p_2}{\sqrt{\al {p_2}^2 + \be p_2 + \symbga}}\\
&= \dfrac{1}{\pi} \left( V_1 \mcN_A(\al,\be,\symbga) + V_2 \mcN_B(\al,\be,\symbga,2) + V_3 \mcN_B(\al,\be,\symbga,2R) \right),
\end{align*} where
\begin{align*}
V_1 := & -(2 {s_1}-1) (R {s_2}-R+{s_2}) \\
V_2 := & -(-2 R {s_1} {s_2}+2 R {s_1}+R {s_2}-R-2 {s_1} {s_2}+{s_2}) \\
V_3 := & -(-2 R^2 {s_1} {s_2}+2 R^2 {s_1}+R^2 {s_2}-R^2-2 R {s_1} {s_2}+R {s_2})\\
\al := &\, 4 \left({s_1}^2-{s_1}+({s_2}-1) {s_2}\right)^2 \\
\be := & -8 (1 + R) (-{s_1} + {s_1}^2 + (-1 + {s_2}) {s_2})^2 \\
\symbga := & -R^2 (1-2 {s_1})^2 ({s_2}-1)^2+2 R \left(8 {s_1}^4-16 {s_1}^3+4 {s_1}^2 \left(3 {s_2}^2-3 {s_2}+2\right) \right. \\& \left. -12 {s_1} ({s_2}-1) {s_2}+{s_2} \left(8 {s_2}^3-16 {s_2}^2+7 {s_2}+1\right)\right)-(1-2 {s_1})^2 {s_2}^2.
\end{align*} By substituting $\mcN_A(\al,\be,\symbga)$ and $\mcN_B(\al,\be,\symbga,\delta)$ we obtain the desired result. The proof for the singularity $S \times N$ is completely analogous but taking into account that the cases II and V should be exchanged and the same for the cases I and IV. 
\end{proof}

\begin{co}
\label{co:rels}
The two components $(h_1, h_2)$ of the height invariant have an intricate dependence on the four parameters $s_1, s_1, R_1, R_2$ of the system but a very simple relation between each other, namely $h_2=2-h_1$.
\end{co}

We can finally extend our results to the case $R_1>R_2$:

\begin{co}
The transformation $\Psi_3$ extends Theorem \ref{th:semit}, Corollary \ref{co:nff}, Theorem \ref{th:poly} and Theorem \ref{th:height} to the case $R_1 > R_2$.
\end{co}

Note that the condition $R_1 = R_2$ results in a non-simple semitoric system, which lies outside the scope of the present work.

%%%%%%%%%%%%%%%%%%%%%%%%%%%%%%%%%%%%%%%%%%%%%%%%%%%%%%%%%%%%%%%%%%%%%%%%%
%%%%%%%%%%%%%%  new section  %%%%%%%%%%%%%%%%%%%%%%%%%%%%%%%%%%%%%%%%%%%%
%%%%%%%%%%%%%%%%%%%%%%%%%%%%%%%%%%%%%%%%%%%%%%%%%%%%%%%%%%%%%%%%%%%%%%%%%%%

\section*{Acknowledgements}

We would like to thank Joseph Palmer, Holger Dullin, and Marine Fontaine for helpful discussions and Wim Vanroose for sharing his computational resources. The authors have been partially funded by the FWO-EoS project G0H4518N and the UA-BOF project with Antigoon-ID 31722.

%%%%%%%%%%%%%%%%%%%%%%%%%%%%%%%%%%%%%%%%%%%%%%%%%%%%%%%%%%%%%%%%%%%%%%%%%
%%%%%%%%%%%%%%  new section  %%%%%%%%%%%%%%%%%%%%%%%%%%%%%%%%%%%%%%%%%%%%
%%%%%%%%%%%%%%%%%%%%%%%%%%%%%%%%%%%%%%%%%%%%%%%%%%%%%%%%%%%%%%%%%%%%%%%%%%%

\bibliographystyle{abbrv}
\bibliography{HEI-Refs}

{\small
  \noindent
  \\[-0.2cm]
  \textbf{Jaume Alonso}\\
  University of Antwerp\\
  Department of Mathematics\\
  Middelheimlaan 1\\
  B-2020 Antwerpen, Belgium\\
  {\em E\--mail}: \texttt{jaume.alonsofernandez@uantwerpen.be}\\[-0.1cm]
}

{\small
  \noindent
  \textbf{Sonja Hohloch}\\
  University of Antwerp\\
  Department of Mathematics\\
  Middelheimlaan 1\\
  B-2020 Antwerpen, Belgium\\
  {\em E\--mail}: \texttt{sonja.hohloch@uantwerpen.be}
}
\end{document}